\def\isarxiv{1}

\ifdefined\isarxiv
\documentclass[11pt]{article}

\usepackage{amsthm, amsmath, amssymb, graphicx, url}
\usepackage[margin=1in]{geometry}

\else
\documentclass[final,12pt]{colt2023} % Include author names

\fi

\usepackage{latexsym, amscd, amsfonts, mathrsfs, stmaryrd, tikz-cd, mathrsfs, bbm, esint, listings, moreverb, hyperref, pifont, algorithm, algpseudocode, enumitem}

\def\bE{\mathbb{E}}
\def\bP{\mathbb{P}}
\def\bR{\mathbb{R}}
\def\bZ{\mathbb{Z}}

\def\cA{\mathcal{A}}

\def\cF{\mathcal{F}}

\def\cO{\mathcal{O}}
\def\cP{\mathcal{P}}

\def\cX{\mathcal{X}}
\def\cY{\mathcal{Y}}
\def\cZ{\mathcal{Z}}

\DeclareFontFamily{U}{mathx}{}
\DeclareFontShape{U}{mathx}{m}{n}{<-> mathx10}{}
\DeclareSymbolFont{mathx}{U}{mathx}{m}{n}
\DeclareMathAccent{\widehat}{0}{mathx}{"70}
\DeclareMathAccent{\widecheck}{0}{mathx}{"71}

\def\iidsim{\stackrel{\text{i.i.d.}}{\sim}}

\def\wh{\widehat}
\def\wt{\widetilde}
\def\aas{\text{a.a.s.}}

\newcommand{\circnum}[1]{%
  \text{\ding{\the\numexpr #1+191}}%
}

\DeclareMathOperator{\argmax}{arg\,max}

\DeclareMathOperator\Aut{\mathrm{Aut}}

\DeclareMathOperator\BP{\mathrm{BP}}

\DeclareMathOperator\BOT{\mathrm{BOT}}
\DeclareMathOperator\BSC{\mathrm{BSC}}

\DeclareMathOperator\diag{\mathrm{diag}}
\DeclareMathOperator\EC{\mathrm{EC}}

\DeclareMathOperator\FSC{\mathrm{FSC}}

\DeclareMathOperator\Id{\mathrm{Id}}

\DeclareMathOperator\Pois{\mathrm{Pois}}

\DeclareMathOperator\SBM{\mathrm{SBM}}
\DeclareMathOperator\SKL{\mathrm{SKL}}

\DeclareMathOperator\Unif{\mathrm{Unif}}
\DeclareMathOperator\Var{\mathrm{Var}}

\ifdefined\isarxiv
\newtheorem{theorem}{Theorem}
\newtheorem{lemma}[theorem]{Lemma}
\newtheorem{proposition}[theorem]{Proposition}

\theoremstyle{definition}
\newtheorem{definition}[theorem]{Definition}

\fi

\usepackage{times}
\begin{document}

\ifdefined\isarxiv
\title{Uniqueness of BP fixed point for the Potts model and applications to community detection}
\date{}
\author{
Yuzhou Gu\thanks{\texttt{yuzhougu@mit.edu}. MIT.}
\and
Yury Polyanskiy\thanks{\texttt{yp@mit.edu}. MIT.}
}
\else
\title[Uniqueness of BP fixed point for the Potts model]{Uniqueness of BP fixed point for the Potts model and applications to community detection}
\coltauthor{%
 \Name{Yuzhou Gu} \Email{yuzhougu@mit.edu}\\
 \addr Masaschusetts Institute of Technology
 \AND
 \Name{Yury Polyanskiy} \Email{yp@mit.edu}\\
 \addr Masaschusetts Institute of Technology%
}
\fi

\maketitle

\begin{abstract}%
  In the study of sparse stochastic block models (SBMs) one often needs to analyze a distributional recursion, known as the belief propagation (BP) recursion. Uniqueness of the fixed point of this recursion implies several results about the SBM, including optimal recovery algorithms for SBM \cite{mossel2016belief} and SBM with side information \cite{mossel2016local}, and a formula for SBM mutual information \cite{abbe2021stochastic}. The 2-community case corresponds to an Ising model, for which~\cite{yu2022ising} established uniqueness for all cases.

  In this paper we analyze the $q$-ary Potts model, i.e., broadcasting of $q$-ary spins on a Galton-Watson tree with expected offspring degree $d$ through Potts channels with second-largest eigenvalue $\lambda$. We allow the intermediate vertices to be observed through noisy channels (side information). We prove that BP uniqueness holds with and without side information when $d\lambda^2 \ge 1 + C \max\{\lambda, q^{-1}\}\log q$ for some absolute constant $C>0$ independent of $q,\lambda,d$. For large $q$ and $\lambda = o(1/\log q)$, this is asymptotically achieving the Kesten-Stigum threshold $d\lambda^2=1$.
  %For the low SNR regime, we prove that these properties hold whenever $d\lambda^2 < q^{-2}$.
  These results imply mutual information formulas and optimal recovery algorithms for the $q$-community SBM in the corresponding ranges.

  % In this paper, we study uniqueness of BP fixed point and boundary irrelevance for the $q$-ary Potts model. We prove that these properties hold for a wide range of parameters. In particular, for ferromagnetic Potts models, these properties hold whenever signal-to-noise-ratio (SNR) is outside the range $[f(q), g(q)]$, where $f(q) = 1/q^2$, $g(q) = \Theta(\log q)$ as $q\to \infty$. This improves over the previous best result \cite{chin2020optimal} which proved Uniqueness of BP Fixed Point for SNR larger than $\poly(q)$.

  For $q\ge 4$, \cite{sly2011reconstruction,mossel2022exact} showed that there exist choices of $q,\lambda,d$ below Kesten-Stigum (i.e. $d\lambda^2 < 1$) but reconstruction is possible.
  Somewhat surprisingly, we show that in such regimes BP uniqueness does not hold at least in the presence of weak side information.

  Our technical tool is a theory of $q$-ary symmetric channels, that we initiate here, generalizing the classical and widely-utilized information-theoretic characterization of BMS (binary memoryless symmetric) channels.
\end{abstract}

\ifdefined\isarxiv
\else
\begin{keywords}%
  stochastic block model, Potts model, broadcasting on trees, cavity equation, belief propagation, boundary irrelevance, $q$-ary symmetric channels
\end{keywords}
\fi

% \tableofcontents

\section{Introduction}
\paragraph{Stochastic block model.}
The stochastic block model (SBM) is a random graph model with community structures. It has a rich set of results and phenomena, investigated in the last decade (see \cite{abbe2017community} for a survey).
In this paper, we focus on the sparse symmetric multi-community case. The model has four parameters: $n\in \bZ_{\ge 1}$, the number of vertices; $q\in \bZ_{\ge 2}$, the number of communities; $a,b\in \bR_{\ge 0}$, parameters controlling edge probabilities.
The model is defined as follows.
First, we assign a random label (community) $X_i \sim \Unif([q])$ i.i.d~for $i\in V=[n]$. Then a random graph $G=(V, E)$ is generated, where $(i,j)\in E$ with probability $\frac an$ if $X_i=X_j$, and with probability $\frac bn$ if $X_i \ne X_j$, independently for all $(i,j)\in \binom V2$.
When $a>b$, we say the model is assortative. When $a<b$, we say the model is disassortative.

For the SBM, an important problem is weak recovery. We say the model admits weak recovery if there exists an estimator $\wh X(G) \in [q]^V$ such that
\begin{align}
  &\lim_{n\to \infty} \frac 1n \bE d_H(X, \wh X(G)) < 1-\frac 1q, \\
  \text{where}\quad &d_H(X,Y) := \min_{\tau \in \Aut([q])} \sum_{i\in [n]} \mathbbm{1}\{X_i \ne \tau(Y_i)\}.
\end{align}
\cite{decelle2011asymptotic} conjectured that the (algorithmic) weak recovery threshold is at the Kesten-Stigum threshold, and there is an information-computation gap for $q\ge 5$. The positive (algorithmic) part of their conjecture has been established by a series of works \cite{massoulie2014community,mossel2018proof,abbe2016achieving,abbe2015community,abbe2015detection,abbe2018proof,bordenave2015non,stephan2019robustness}, in a very general sense allowing asymmetric communities.
\cite{abbe2015community,abbe2015detection,abbe2018proof} gave inefficient reconstruction algorithms below the Kesten-Stigum threshold, for disassortative models with $q\ge 4$ and assortative models with $q\ge 5$, giving evidence for the information-computation gap.
The informational weak recovery threshold has been established in special cases by \cite{mossel2015reconstruction,mossel2018proof} ($q=2$), \cite{mossel2022exact} ($q=3,4$ assuming large enough degree), \cite{coja2017information} (general $q$, disassortative).
For the assortative case, the informational weak recovery threshold for $q=3,4$ with small degree or $q\ge 5$ is still open, despite some partial progress \cite{banks2016information,gu2020non}.
The information-computation gap is also wide open.

When weak recovery is possible, the natural follow-up question is to determine the optimal recovery accuracy
\begin{align}
  \sup_{\wh X=\wh X(G)} \lim_{n\to \infty} \bE \left[1-\frac 1n d_H(X, \wh X(G))\right].
\end{align}
\cite{decelle2011asymptotic} conjectured that the belief propagation algorithm is optimal.
The conjecture is not proved yet but significant progress has been made: a series of works \cite{mossel2016belief,abbe2021stochastic,yu2022ising} established an optimal recovery algorithm for the symmetric two-community SBM; \cite{chin2020optimal,chin2021optimal} gave optimal recovery algorithms for the general case (general $q$, not necessarily symmetric) under the condition that SNR is large enough.

A fundamental quantity of the stochastic block model is its (normalized) mutual information
\begin{align}
  \lim_{n\to \infty} \frac 1n I(X; G).
\end{align}
\cite{coja2017information} proved a mutual information formula for the disassortative model (general $q$).
\cite{abbe2021stochastic,yu2022ising} proved a mutual information formula for the $q=2$ case.
\cite{dominguez2022mutual} conjectured another formula for the $q=2$ case and proved a matching lower bound. It is not known whether their conjectured formula is equivalent to the one proved by \cite{abbe2021stochastic,yu2022ising}.

\paragraph{Broadcasting on trees.}
The stochastic block model has a close relationship with the broadcasting on trees (BOT) model. The reason is that in SBM, the local neighborhood  of a random vertex converges (in the sense of local weak convergence) to a Galton-Watson tree with Poisson offspring distribution. Therefore, properties of BOT can often imply corresponding results on SBM.

For the symmetric $q$-SBM, the corresponding model is the Potts model, i.e., the BOT model whose broadcasting channels are Potts channels. This model has three parameters: $q\in \bZ_{\ge 2}$, the number of colors; $\lambda \in [-\frac 1{q-1}, 1]$, edge correlation strength; $d\in \bZ_{\ge 0}$, expected offspring.
The Potts model is defined as follows.
Let $T$ be a regular tree of offspring $d$ or a Galton-Watson tree with offspring distribution $\Pois(d)$ (Poisson distribution with expectation $d$).
Let $\rho$ be the root of $T$.
We assign to every vertex $v$ a color $\sigma_v \in [q]$ according to the following process:
(1) $\sigma_\rho \sim \Unif([q])$;
(2) given $\sigma_u$, colors of children of $u$ are $\iidsim P_\lambda(\cdot | \sigma_u)$,
where $P_\lambda$ is the Potts channel defined as
$P_\lambda(j|i) = \lambda \mathbbm{1}\{i=j\} + \frac{1-\lambda}q.$

An important problem on BOT is the reconstruction problem, asking whether we can gain any non-trivial information about the root given observation of far away vertices.
We say the model admits reconstruction if
\begin{align}
  \lim_{k\to \infty} I(\sigma_\rho ; \sigma_{L_k} | T_k) > 0,
\end{align}
where $L_k$ stands for the set of vertices at distance $k$ to the root $\rho$.
We say the model admits non-reconstruction if the limit is zero.
It is known that non-reconstruction results for the Potts model imply impossibility of weak recovery for the corresponding SBM, but the other direction does not hold: in the case $a=0$, there is a gap of factor $2$ (as $q\to \infty$) between the BOT reconstruction threshold and the SBM weak recovery threshold.

The reconstruction problem on trees has been studied a lot under many different settings, e.g., \cite{bleher1995purity,evans2000broadcasting,mossel2001reconstruction,mossel2003information,mezard2006reconstruction,borgs2006kesten,bhatnagar2010reconstruction,sly2009reconstruction,kulske2009symmetric,liu2019large,gu2020non,mossel2022exact}. For the Potts model, it is known \cite{sly2011reconstruction,mossel2022exact} that the Kesten-Stigum threshold \cite{kesten1966additional} $d\lambda^2=1$ is tight (i.e., equal to the reconstruction threshold) for $q=3,4$ when $d$ is large enough, and is not tight when $q\ge 5$ or when $q=4$, $\lambda<0$ and $d$ is small enough.

\paragraph{Belief propagation.}
Belief propagation is a powerful tool for studying the BOT model. It is usually described as an algorithm for computing posterior distribution of vertex colors given observation. Here we take an information-theoretic point of view and describe BP in terms of constructing communication channels.

We view the BOT model as an information channel from the root color to the observation.
Let $M_k$ denote the channel $\sigma_\rho \mapsto (T_k, \sigma_{L_k})$.%, where $L_k$ denotes the set of vertices at distance $k$ to the root $\rho$.
Then $(M_k)_{k\in \bZ_{\ge 0}}$ satisfies the following recursion, which we call belief propagation recursion:
\begin{align}
  M_{k+1} = \bE_b (M_k \circ P_\lambda)^{\star b}
\end{align}
where $b$ follows the branching number distribution (constant in the regular tree case, $\Pois(d)$ in the Poisson tree case), and $(\cdot)^{\star b}$ denotes $\star$-convolution power. % (see Section~\ref{sec:fms} for more information).
Let $\BP$ be the operator
\begin{align} \label{eqn:bp-operator}
  \BP(M) := \bE_b (M \circ P_\lambda)^{\star b}
\end{align}
defined on the space of information channels with input alphabet $[q]$. Due to symmetry in colors, we can regard $\BP$ as an operator on the space of FMS channels (see Section~\ref{sec:fms}).
In terms of the $\BP$ operator, the reconstruction problem can be rephrased as whether the limit channel $\BP^\infty(\Id) := \lim_{n\to \infty} \BP^n(\Id)$ (where $\Id$ stands for the identity channel $\Id(y|x) = \mathbbm{1}\{x=y\}$) is trivial or not.
The problem of optimal recovery for SBM can be reduced to the following problem on trees: whether the limit
\begin{align}
  \lim_{n\to \infty} I(\sigma_\rho; \omega_{L_k} | T_k)
\end{align}
where $\omega$ is the observation of $\sigma$ through a non-trivial channel $W$, stays the same for any non-trivial FMS channel $W$. Therefore, it is important to study the non-trivial fixed points of the $\BP$ operator (the trivial channel is always a fixed point).

\cite{mossel2016belief} proved uniqueness of BP fixed point for $q=2$ and large enough SNR.
\cite{abbe2021stochastic} improved to $q=2$ and SNR $\not \in [1,3.513]$.
\cite{yu2022ising} proved uniqueness of BP fixed point for $q=2$ and any parameter $\lambda,d$, closing the question for binary symmetric models.
For $q\ge 3$, \cite{chin2020optimal} proved a local version of BP uniqueness, i.e., when the initial channel $U$ is close enough to $\Id$, and $d\lambda^2 > C_q$, where $C_q$ is a constant depending on $q$, then $\BP^\infty(U) = \BP^\infty(\Id)$. They did not give asymptotics for $C_q$, but it seems like it is at least polynomial in $q$.
\cite{chin2021optimal} generalized \cite{chin2020optimal} to asymmetric models.

\paragraph{Boundary irrelevance.}
\cite{abbe2021stochastic} reduced the SBM mutual information problem to the boundary irrelevance problem, on a tree model called the broadcasting on trees with survey (BOTS) model. In the BOTS model, we observe label of every vertex through a noisy FMS channel $W$ (called the survey).
We say the model admits boundary irrelevance with respect to $W$ if
\begin{align}
  \lim_{k\to \infty} I(\sigma_\rho ; \sigma_{L_k} | T_k, \omega_{T_k})=0,
\end{align}
where $T_k$ is the set of all vertices within distance at most $k$ to the root, and $\omega$ is the observation of $\sigma$ through $W$.
We say the model admits boundary irrelevance if the model admits boundary irrelevance with respect to all erasure channels $\EC_\epsilon$ with $0\le \epsilon<1$.
Boundary irrelevance is equivalent to the condition that the operator
\begin{align} \label{eqn:bp-w-operator}
  \BP_W(M) := \left(\bE_b (M\circ P_\lambda)^{\star b}\right) \star W
\end{align}
has a unique fixed point in the space of FMS channels.
Because $\BP$ and $\BP_W$ have very similar forms, the boundary irrelevance problem has a close relationship with the problem of uniqueness of BP fixed point. Indeed, these two problems can be solved using the same method.

% \paragraph{Symmetric channels.}
% One technical tool we develop is a theory of $q$-ary input symmetric channels (Section~\ref{sec:fms}).
% The topic of symmetric channels is classic in information theory (see e.g.~\cite[Section 19.4]{polyanskiy2023information}). The binary-input case, i.e., the BMS channels, has been studied extensively and found numerous applications (see e.g.,~\cite[Chapter 4]{richardson2008modern}).
% However, the theory of comparison between $q$-ary symmetric channels is much less developed before the current work, with only example known to us being~\cite{makur2018comparison}.

\paragraph{Our results.}
Our first main result is uniqueness of BP fixed point and boundary irrelevance for a wide range of parameters. In fact, we prove the stronger result that the BP fixed point satisfies global stability.
\begin{theorem}[Uniqueness of BP fixed point and boundary irrelevance] \label{thm:bi-imprecise}
  There exists an absolute constant $C>0$ such that the following statement holds.
  Consider the $q$-ary Potts model with broadcasting channel $P_\lambda$ on a regular tree or a Poisson tree with expected offspring $d$.
  If either $d\lambda^2 < q^{-2}$ or $d\lambda^2 > 1 + C \max\{\lambda, q^{-1}\} \log q$,
  then boundary irrelevance holds.
  That is, for any non-trivial $q$-FMS survey channel $W$, we have
  \begin{align}
    \lim_{k\to \infty} I(\sigma_\rho; \sigma_{L_k} | T_k, \omega_{T_k}) = 0.
  \end{align}
  Furthermore, under the same conditions, stability of BP fixed point holds, i.e., for any non-trivial $q$-FMS channel $P$, $\BP^k(P)$ and $\BP^k(\Id)$ converge weakly to the same limit as $k\to \infty$.
\end{theorem}
For a more precise statement, see Theorem~\ref{thm:bi-low-snr} and Theorem~\ref{thm:bi-high-snr}.
We generalize this result to asymmetric initial channels in Prop.~\ref{prop:bp-fixed-point-full-rank} and Prop.~\ref{prop:bp-prime-fixed-point-asymm}.
Compared with \cite{chin2020optimal}, our result has a much weaker assumption on the initial channel, and works for a much larger region of $(\lambda,d)$ (at least when $q\to \infty$).

Our second main result is that boundary irrelevance does not hold between the reconstruction threshold and the Kesten-Stigum threshold.
\begin{theorem}[Boundary irrelevance does not always hold] \label{thm:non-bi}
  Consider the $q$-ary Potts model with broadcasting channel $P_\lambda$ on a regular tree or a Poisson tree with expected offspring $d$.
  If $d\lambda^2 < 1$ and reconstruction is possible, then boundary irrelevance does not hold for weak enough survey channel.
  That is, there exists $\epsilon > 0$ such that for any FMS survey channel $W$ with $C_{\chi^2}(W) \le \epsilon$, we have
  $
    \lim_{k\to \infty} I(\sigma_\rho; \sigma_{L_k} | T_k, \omega_{T_k}) > 0.
  $
\end{theorem}
By \cite{sly2011reconstruction,mossel2022exact}, for any $q\ge 4$, there exist choices of $\lambda,d$ satisfying the assumption in Theorem~\ref{thm:non-bi}.
Therefore for any fixed $q\ge 4$, boundary irrelevance does not always hold, in contrast to the binary case where boundary irrelevance always holds \cite{yu2022ising}.

\paragraph{Applications.}
Main applications of uniqueness of BP fixed point and boundary irrelevance include a mutual information formula and an optimal recovery algorithm.
\begin{theorem}[Mutual information formula] \label{thm:sbm-mutual-info}
  Let $(X,G)\sim \SBM(n,q,\frac an,\frac bn)$.
  Let $d = \frac {a+(q-1)b}q$ and $\lambda = \frac{a-b}{a+(q-1)b}$.
  Let $(T,\sigma)$ be the Potts model with broadcasting channel $P_\lambda$ on a Poisson tree with expected offspring $d$. Let $\rho$ be the root of $T$, $L_k$ be the set of vertices at distance $k$ to $\rho$, $T_k$ be the set of vertices at distance $\le k$ to $\rho$.
  If $(q,\lambda,d)$ satisfies~\eqref{eqn:thm:bi-low-snr-cond} or~\eqref{eqn:thm:bi-high-snr-cond}, then we have
  \begin{align}
    \lim_{n\to \infty} \frac 1n I(X;G) = \int_0^1 \lim_{k\to \infty} I(\sigma_\rho; \omega^\epsilon_{T_k\backslash \rho} | T_k) d\epsilon,
  \end{align}
  where $\omega^\epsilon$ denotes observation through survey channel $\EC_\epsilon$, the erasure channel with erasing probability $\epsilon$.
\end{theorem}

For SBM with side information, boundary irrelevance immediately implies that the local belief propagation algorithm is optimal.
\begin{theorem}[Optimal recovery for SBM with side information] \label{thm:opt-rec-survey}
  Work under the same setting as Theorem~\ref{thm:sbm-mutual-info}.
  Suppose that in addition to $G$, we observe side information $Y_v \sim W(\cdot | X_v)$ for all $v\in V$, where $W$ is some non-trivial FMS channel.
  If $(q,\lambda,d,W)$ satisfies~\eqref{eqn:thm:bi-low-snr-cond} or~\eqref{eqn:thm:bi-w-high-snr-cond}, then belief propagation (Algorithm~\ref{alg:opt-rec-survey}) achieves the optimal recovery accuracy of
  \begin{align}
    1-\lim_{k\to \infty} P_e(\sigma_\rho | T_k, \omega_{T_k}).
  \end{align}
\end{theorem}

For vanilla SBM, uniqueness of BP fixed point implications optimal recovery, given an initial recovery algorithm with nice accuracy guarantees.
\begin{theorem}[Optimal recovery for SBM] \label{thm:opt-rec-vanilla}
  Work under the same setting as Theorem~\ref{thm:sbm-mutual-info}.
  Suppose $d\lambda^2>1$ and $(q,\lambda,d)$ satisfies~\eqref{eqn:thm:bi-high-snr-cond}.
  Suppose there is an algorithm $\cA$ and a constant $\epsilon>0$ (not depending on $n$)  such that with probability $1-o(1)$, the empirical transition matrix $F\in \bR^{q\times q}$ defined as
  \begin{align}
    F_{i,j} := \frac{\#\{v\in V: X_v=i, \wh X_v = j\}}{\#\{v\in V: X_v = i\}}, \qquad \wh X := \cA(G)
  \end{align}
  satisfies
  \begin{enumerate}[label=(\arabic*)]
    \item $\|F^\top \mathbbm{1} - \mathbbm{1}\|_{\infty} = o(1)$; \label{item:thm-opt-rec-vanilla-req-1}
    \item $\sigma_{\min}(F) > \epsilon$, where $\sigma_{\min}$ denotes the smallest singular value; \label{item:thm-opt-rec-vanilla-req-2}
    \item there exists a permutation $\tau\in \Aut([q])$ such that $F_{\tau(i),i} > F_{\tau(i),j}+\epsilon$ for all $i\ne j\in [q]$. \label{item:thm-opt-rec-vanilla-req-3}
  \end{enumerate}
  (Note that we do not asssume $F$ stays the same for different calls to $\cA$.)

  Then there is an algorithm (Algorithm~\ref{alg:opt-rec-vanilla}) achieving the optimal recovery accuracy of
  \begin{align}
    1-\lim_{k\to \infty} P_e(\sigma_\rho | T_k, \sigma_{L_k}).
  \end{align}
\end{theorem}
Our assumption on the initial recovery algorithm is a generalization of the one used in the $q=2$ case by \cite{mossel2016belief}. Our assumption is much weaker than the one of \cite{chin2020optimal}, which requires the initial point to be close enough to $\Id$ and seems unlikely to hold near the KS threshold.
% The theoretical guarantees of the initial recovery algorithm provided by previous works (that achieve the Kesten-Stigum threshold) \cite{abbe2018proof,stephan2019robustness} do not seem to be enough for our purpose. There are several different ways to formulate the initial point requirement. The one we state here is weaker than \cite{chin2020optimal} (which required the initial point to be close enough to $\Id$, which seems unlikely to hold near the KS threshold), and a generalization of the requirement in the $q=2$ case used by \cite{mossel2016belief}.
In comparison, our initial point assumption seems more likely to hold near the Kesten-Stigum threshold. For example, it is plausible that a balanced algorithm would achieve the empirical transition matrix $F$ to be close to $P_\lambda$ for some $|\lambda|=\Omega(1)$.

\paragraph{Our technique.}
For the positive result (Theorem~\ref{thm:bi-imprecise}), we generalize the degradation method of \cite{abbe2021stochastic} to $q$-ary symmetric channels.
In this method, we find suitable potential functions $\Phi$ on the space of FMS channels, such that for two channels $M,\wt M$ related by degradation ($\wt M\le_{\deg} M$), we have (1) $\Phi(M)-\Phi(\wt M)$ contracts to $0$ under iterations of $\BP$ (2) if $\Phi(M) = \Phi(\wt M)$, then $M=\wt M$.
This shows that the limit channels $\BP^\infty(M)$ and $\BP^\infty(\wt M)$ are equal.

To carry out this method, we develop a theory of $q$-ary symmetric channels, generalizing the classical theory of binary memoryless symmetric (BMS) channels. We show that $q$-ary symmetric channels are equivalent to symmetric distributions on the probability simplex $\cP([q])$, and degradation relationship has a coupling characterization under the distribution interpretation.

We remark that \cite{yu2022ising} also uses channel degradation but in a very different way. For more discussions see Section~\ref{sec:discuss}.

For the negative result (Theorem~\ref{thm:non-bi}), we show that when the survey channel $W$ is weak enough, the limit channel $\BP_W^\infty$ can be arbitrarily weak.
We prove this by studying behavior of $\chi^2$-capacity under BP recursion.
One difficulty of working with $\chi^2$-capacity is that in general they are not subadditive under $\star$-convolution.
Subadditivity is a very desirable property when studying BP recursion, and holds for KL capacity (folklore) and symmetric KL (SKL) capacity \cite{kulske2009symmetric}. For $\chi^2$-capacity, subadditivity is true for BMS channels \cite{abbe2019subadditivity}, but there are counterexamples when $q\ge 3$.
Nevertheless, we establish a local subadditivity result, i.e., $\chi^2$-capacity is almost subadditive under $\star$-convolution for weak enough channels.
% We use $\chi^2$-capacity to characterize the strength of a channel.
% One difficulty of working with $\chi^2$-capacity is that unlike KL capacity or symmetric KL (SKL) capacity, they are not subadditive under $\star$-convolution.
% To deal with this problem, we establish a local subadditivity result (Lemma~\ref{lem:subadd}).
% That is, $\chi^2$-capacity is almost subadditive for weak enough channels.

Mutual information formula (Theorem~\ref{thm:sbm-mutual-info}) and optimal recovery for SBM with side information (Theorem~\ref{thm:opt-rec-survey}) are direct consequences of boundary irrelevance (Theorem~\ref{thm:bi-imprecise}), generalizing \cite{abbe2021stochastic} and \cite{mossel2016local} respectively.
For optimal recovery for the vanilla SBM (Theorem~\ref{thm:opt-rec-vanilla}), we need to handle asymmetric initial points (Section~\ref{sec:asymm}), and suitably generalize the algorithm in \cite{mossel2016belief}.

\paragraph{Structure of the paper.}
In Section~\ref{sec:fms}, we establish a theory of FMS channels.
In Section~\ref{sec:bi}, we prove Theorem~\ref{thm:bi-imprecise}, boundary irrelevance and uniqueness of BP fixed point for a wide range of parameters.
In Section~\ref{sec:non-bi}, we prove Theorem~\ref{thm:non-bi}, that boundary irrelevance does not hold between the reconstruction threshold and the Kesten-Stigum threshold.

In Section~\ref{sec:chan-lim}, we discuss limits of information channels.
In Section~\ref{sec:proof-fms}, we give missing proofs in Section~\ref{sec:fms}.
In Section~\ref{sec:proof-bi}, we give missing proofs in Section~\ref{sec:bi}.
In Section~\ref{sec:proof-non-bi}, we give missing proofs in Section~\ref{sec:non-bi}, including the local subadditivity of $\chi^2$-information (Lemma~\ref{lem:subadd}).
In Section~\ref{sec:asymm}, we discuss asymmetric fixed points of the $\BP$ operator.
In Section~\ref{sec:mutual-info}, we prove Theorem~\ref{thm:sbm-mutual-info}, SBM mutual information formula.
In Section~\ref{sec:opt-rec}, we prove Theorem~\ref{thm:opt-rec-survey} and~\ref{thm:opt-rec-vanilla}, optimal recovery algorithms.
In Section~\ref{sec:discuss}, we discuss several further directions.

\section{FMS Channels} \label{sec:fms}
In this section we introduce $q$-ary fully memoryless symmetric ($q$-FMS) channels.\footnote{Here ``fully'' modifies ``symmetric'', and indicates that the symmetry group of the channel is the full symmetric group $\Aut(\cX)$ as opposed to a subgroup.} They are a generalization of binary memoryless symmetric (BMS) channels to $q$-ary input alphabets.
For background on BMS channels, see e.g., \cite{richardson2008modern}.
\begin{definition}[Fully memoryless symmetric (FMS) channels] \label{defn:fms-channel}
	A $q$-ary fully memoryless symmetric ($q$-FMS) channel (or an FMS channel when $q$ is obvious from context) is a channel $P: \cX \to \cY$ with input alphabet $\cX = [q]$ such that
	there exists a group homomorphism $\iota: \Aut(\cX) \to \Aut(\cY)$ such that
	for any measurable $E\subseteq \cY$, we have
	\begin{align}
		P(E | x) = P(\iota(\tau) E | \tau(x))
  \end{align}
	for all $x\in \cX$, $\tau \in \Aut(\cX)$.
  Here $\Aut(\cX)$ denotes the symmetry group (also known as the automorphism group) of $\cX$.
\end{definition}
By definition, $2$-FMS channels are exactly BMS channels.

We remark that $q$-FMS channels are a special case of input-symmetric channels (see e.g.,~\cite[Chapter 19]{polyanskiy2023information}) whose group of symmetries is the whole $\Aut(\cX)$.
\cite{makur2018comparison} studied comparison between $q$-ary symmetric channels, but their definition of symmetric channels is quite different from ours.

The BSC mixture representation of BMS channels has been useful in proving results about BMS channels. Therefore it is desirable to generalize this theory to FMS channels.
We define fully symmetric channels (FSCs), which generalize BSCs, and will serve as basic building blocks for FMS channels.
\begin{definition} \label{defn:fsc-channel}
	Let $\cX = [q]$, $\cY = \Aut(\cX)$. For $\pi \in \cP(\cX)/\Aut(\cX)$, define channel $\FSC_\pi: \cX \to \cY$ as
	\begin{align}
		\FSC_\pi(\tau | i) = \frac 1{(q-1)!}\pi_{\tau^{-1}(i)} \qquad \forall i\in \cX, \tau\in \Aut(\cX),
	\end{align}
	where $\Aut(\cX)$ acts on $\cY$ by left multiplication.
\end{definition}
We can verify that
\begin{align}
  \FSC_\pi(\eta \tau | \eta(i)) = \frac{1}{(q-1)!} \pi_{(\eta\tau)^{-1}(\eta(i))} = \frac 1{(q-1)!} \pi_{\tau^{-1}(i)} = \FSC_\pi(\tau | i)
\end{align}
for $i\in \cX$, $\eta,\tau\in \Aut(\cX)$. So FSCs are examples of FMS channels.

One of the most basic relationships between two channels is degradation.
\begin{definition} \label{defn:fsc-degradation}
  Let $P: \cX \to \cY$ ad $Q: \cX \to \cZ$ be two channels with the same input alphabet.
  We say $P$ is a degradation of $Q$ (denoted $P \le_{\deg} Q$) if there exists a channel $R: \cZ \to \cY$ respecting $\Aut(\cX)$ action such that $P = R\circ Q$.
  We say $P$ and $Q$ are equivalent if $P \le_{\deg} Q$ and $Q \le_{\deg} P$.
\end{definition}
In other words, $P$ is a degradation of $Q$ if we can simulate $P$ by postprocessing the output of $Q$.

In the binary case, all BMS channels are equivalent to mixtures of BSCs (see e.g.,~\cite{richardson2008modern}). We generalize this result to $q$-ary FMS channels.
\begin{proposition}[Structure of FMS channels] \label{prop:fms-mixture}
  Every FMS channel is equivalent to a mixture of FSCs, i.e., every FMS channel $P:\cX\to \cY$ is equivalent to a channel
	$X\to (\pi, Z)$ where $\pi \sim P_\pi\in \cP(\cP(\cX)/\Aut(\cX))$ is independent of $X$,
	and $Z \sim \FSC_\pi(\cdot | X)$ conditioned on $\pi$ and $X$.
	Furthermore, $P_\pi$ is uniquely deteremined by $P$.
\end{proposition}
Proof is deferred to Section~\ref{sec:proof-fms}.
In the setting of the above proposition, we call $\pi$ the $\pi$-component of $P$, and $P_\pi$ the $\pi$-distribution of $P$.
% We often use the convention that elements $\pi\in\cP(\cX)/\Aut(\cX)$ satisfy $\pi_1\ge \cdots\ge \pi_q$.
Prop.~\ref{prop:fms-mixture} establishes an equivalence between an FMS channel and a probability distribution on $\cP(\cX)/\Aut(\cX)$, which maps an FMS channel to its $\pi$-distribution.
In particular, the $\pi$-distribution is an invariant property of an FMS channel under equivalence.
For BMS channels we usually denote the $\pi$-component as a single number $\Delta\in [0, \frac 12]$, and call it the $\Delta$-component.

Degradation has a nice characterization in terms of the $\pi$-components.
\begin{proposition} \label{prop:fms-deg-couple}
	Let $P, Q$ be two FMS channels, and $\pi_P$ and $\pi_Q$ be their $\pi$-components.
	Then $P\le_{\deg} Q$ if and only if there exists a coupling between $\pi_P$ and $\pi_Q$ such that
	\begin{align}
		\pi \le_m \bE[\pi_Q | \pi_P = \pi] \qquad \forall \pi\in \cP(\cX)/\Aut(\cX),
	\end{align}
	where $\le_m$ denotes majorization (see e.g., \cite[2.18]{hardy1934inequalities}).
  We use the convention that elements $\pi\in\cP(\cX)/\Aut(\cX)$ are non-increasing so that the expectation is well-defined.
\end{proposition}
Proof is deferred to Section~\ref{sec:proof-fms}.

There are different ways to construct new FMS channels from given FMS channels.
In this paper we focus on two ways: composition with Potts channels and $\star$-convolution.

Fix $q\ge 2$. For $\lambda \in [-\frac{1}{q-1}, 1]$, define Potts channel $P_\lambda: [q] \to [q]$ as
$
  P_\lambda(y | x) = \lambda \mathbbm{1}\{x=y\} + \frac{1-\lambda}q
$
for $x,y\in [q]$.
Then given any $q$-FMS channel $P$, $P\circ P_\lambda$ is also a $q$-FMS channel.
Furthermore, the $\pi$-distribution of $P\circ P_\lambda$ is $f_*(P_\pi)$, where $P_\pi$ is the $\pi$-distrbution of $P$, $f(\pi) = \lambda \pi + \frac {1-\lambda} q$, and $f_*$ is the induced pushforward map.

Given two channels $P: \cX\to \cY$, $Q: \cY\to \cZ$, their $\star$-convolution $P\star Q: \cX \to \cY \times \cZ$ is defined by letting $P$ and $Q$ acting on the same input independently.
When $P$ and $Q$ are $q$-FMS channels, $P\star Q$ has a natural $q$-FMS structure.
If the $\pi$-component of $P$ (resp.~$Q$) has distribution $P_\pi$ (resp.~$Q_\pi$), then the $\pi$-component of $P\star Q$ has distribution
\begin{align}
  &\bE_{\substack{\pi \sim P_\pi \\ \pi' \sim Q_\pi}}\left[ \sum_{\tau \in \Aut([q])} \left(\frac 1{(q-1)!}\sum_{i\in [q]} \pi_i \pi'_{\tau(i)}\right) \mathbbm{1}_{\pi \star_\tau \pi'}\right]
  \label{eqn:star-formula}\\
  \text{where}\quad &\pi \star_\tau \pi' := \left(\frac {\pi_i \pi'_{\tau(i)}}{\sum_{j\in [q]} \pi_j \pi'_{\tau(j)}}\right)_{i\in [q]} \in \cP([q])/\Aut([q])
\end{align}
and $\mathbbm{1}_\theta \in \cP(\cP([q])/\Aut([q]))$ denotes the point distribution at $\theta\in \cP([q])/\Aut([q])$.
We use $M^{\star b}$ to denote the $b$-th $\star$-power: $M^{\star 0} = \Id$ and $M^{\star b} = M^{\star (b-1)} \star M$.

Given any $q$-FMS channel $P$, we can restrict the input alphabet to get a $q'$-FMS for $q'\le q$. (Because of symmetry, the restricted channel is unique up to channel equivalence no matter what size-$q'$ subset we choose.)
In this paper we only use the case $q'=2$, i.e., restrict to a BMS channel.
We use $P^R$ to denote the restricted BMS channel.\footnote{Here ``$R$'' stands for ``restriction''.}

We study the behavior of information measures under belief propagation. The following information measures are particularly useful.
\begin{definition} \label{defn:fms-info-measure}
  Let $P$ be a $q$-FMS channel and $\pi$ be its $\pi$-component. We define the following quantities.
  \begin{align*}
    P_e(P) &= \bE \min\{1-\pi_i : i\in [q]\},  \tag{probability of error} \\
    C(P) &= \log q - \bE \sum_{i\in [q]} \pi_i \log \frac 1{\pi_i}, \tag{capacity} \\
    C_{\chi^2}(P) &= \bE \left[q \sum_{i\in [q]} \pi_i^2 -1\right], \tag{$\chi^2$-capacity} \\
    C_{\SKL}(P) &= \bE \left[\sum_{i\in [q]} \left(\pi_i - \frac 1q\right) \log(\pi_i) \right]. \tag{SKL capacity}
  \end{align*}
  % \begin{align*}
  %   P_e(P) &= \bE  \min\{\pi_i : i\in [q]\}  \tag{probability of error} \\
  %   C(P) &= \bE D(\pi \| \Unif([q])) \tag{capacity} \\
  %   &= \log q - \bE \sum_{i\in [q]} \pi_i \log \frac 1{\pi_i} \\
  %   C_{\chi^2}(P) &= \bE \chi^2 (\pi \| \Unif([q])) \tag{$\chi^2$-capacity} \\
  %   &= \bE \left[q \sum_{i\in [q]} \pi_i^2 -1\right] \\
  %   C_{\SKL}(P) &= \bE D_{\SKL}(\pi \| \Unif(q))  \tag{SKL capacity}\\
  %   &= \bE \left(D(\pi \| \Unif(q)) + D(\Unif(q) \| \pi)\right) \\
  %   &= \bE \sum_{i\in [q]} (\pi_i - \frac 1q) \log(\pi_i).
  % \end{align*}
  For BMS channels we also define the Bhattacharyya coefficient
  \begin{align}
    Z(P) &= \bE \left[2\sqrt{\Delta(1-\Delta)}\right]
    \tag{Bhattacharyya coefficient}
  \end{align}
  where $\Delta$ is the $\Delta$-component of $P$.
\end{definition}

These information measures respect degradation, as summarized in the next lemma.
\begin{lemma} \label{lem:fms-deg-info-measure}
  Let $P$ and $Q$ be two $q$-FMS channels with $P\le_{\deg} Q$. Then
  $
    P_e(P) \ge P_e(Q), C(P) \le C(Q), C_{\chi^2}(P) \le C_{\chi^2}(Q), C_{\SKL}(P) \le C_{\SKL}(Q).
  $
  If $q=2$ then we also have
  $
    Z(P) \ge Z(Q).
  $
\end{lemma}
\begin{proof}
  By definition of degradation, and data processing inequality for $f$-divergences.
\end{proof}

\section{Uniqueness and boundary irrelevance results} \label{sec:bi}
In this section we prove uniqueness of BP fixed point and boundary irrelevance results for the Potts model for a wide range of parameters.
We consider the $q$-ary Potts model with broadcasting channel $P_\lambda$ on a regular tree or a Poisson tree with expected offspring $d$.

We state two results, one for the low SNR regime and one for the high SNR regime. We define the following constants used in the results.
\begin{definition}
  For $q\in \bZ_{\ge 2}$, $\lambda\in \left[-\frac 1{q-1},1\right]$, $d\ge 0$, we define
  \begin{align}
    C^L(q,\lambda) &:= \sup_{\substack{\pi \in \cP([q]) \\ v \in \mathbbm{1}^\perp \subseteq \bR^q}} \frac{f^L\left(\lambda \pi + \frac {1-\lambda} q, v\right)}{f^L(\pi, v)}, \label{eqn:thm:bi-low-snr-defn-C}\\
    \text{where}~f^L(\pi, v) &:= \left\langle \pi^{-1} + \frac 1q \pi^{-2}, v^2\right\rangle, \label{eqn:thm:bi-low-snr-defn-f}\\
  % \end{align}
  % \begin{align}
    C^H(q,\lambda) &:= \sup_{\substack{\pi \in \cP([q]) \\ v \in \mathbbm{1}^\perp \subseteq \bR^q}} \frac{f^H\left(\lambda \pi + \frac {1-\lambda} q, v\right)}{f^H(\pi, v)}, \label{eqn:thm:bi-high-snr-defn-C}\\
    \text{where}~f^H(\pi, v) &:= \|\pi^{1/4}\|_2^2 \|\pi^{-3/4} v\|_2^2 - \left\langle \pi^{1/4}, \pi^{-3/4}v \right\rangle^2, \label{eqn:thm:bi-high-snr-defn-f}\\
    c^H(q,\lambda,d) &:= \left(\frac 2q + \frac{q-2}q \cdot \frac{d\lambda^2-1}{d\lambda-1}\right)^{-1}.\label{eqn:thm:bi-high-snr-defn-c}
  \end{align}
\end{definition}
We have the following bounds on these constants: $C^L(q,\lambda) \le q^2$ (Prop.~\ref{prop:bi-low-snr-C-bound}), $C^H(q,\lambda) \le q^{5/2}$ (Prop.~\ref{prop:bi-high-snr-C-bound}), $c^H(q,\lambda,d) \ge 1$ (obvious).

\begin{theorem}[Low SNR] \label{thm:bi-low-snr}
  If \begin{align}
    d\lambda^2 C^L(q,\lambda) < 1, \label{eqn:thm:bi-low-snr-cond}
  \end{align}
  where $C^L$ is defined in~\eqref{eqn:thm:bi-low-snr-defn-C}, then boundary irrelevance and stability of BP fixed point hold.
\end{theorem}

\begin{theorem}[High SNR] \label{thm:bi-high-snr}
  If $d\lambda^2>1$ and
  \begin{align}
    d\lambda^2 \exp\left(-c^H(q,\lambda,d)\cdot \frac{d\lambda^2-1}2\right) C^H(q,\lambda) < 1, \label{eqn:thm:bi-high-snr-cond}
  \end{align}
  where $c^H$ is defined in~\eqref{eqn:thm:bi-high-snr-defn-c}, $C^H$ is defined in~\eqref{eqn:thm:bi-high-snr-defn-C},
  then boundary irrelevance and stability of BP fixed point hold.

  Let $W$ be a $q$-FMS channel. If $d\lambda^2>1$ and
  \begin{align}
    d\lambda^2 \exp\left(-c^H(q,\lambda,d)\cdot \frac{d\lambda^2-1}2\right) C^H(q,\lambda) Z(W^R) < 1, \label{eqn:thm:bi-w-high-snr-cond}
  \end{align}
  where $W^R$ denotes the restriction of $W$ to a BMS channel, and $Z$ denotes the Bhattacharyya coefficient,
  then boundary irrelevance holds with repsect to $W$.
\end{theorem}

\begin{proof}[Proof of Theorem~\ref{thm:bi-imprecise} given Theorem~\ref{thm:bi-low-snr} and~\ref{thm:bi-high-snr}]
  We prove the low SNR case and the high SNR case separately.

  \textbf{Low SNR:}
  By Prop.~\ref{prop:bi-low-snr-C-bound}, $C^L(q,\lambda) \le q^2$. If $d \lambda^2 < q^{-2}$, then~\eqref{eqn:thm:bi-low-snr-cond} holds and Theorem~\ref{thm:bi-low-snr} applies.

  \textbf{High SNR:}
  We prove that~\eqref{eqn:thm:bi-high-snr-cond} holds whenever $d\lambda^2 > 1 + 56 \max\{\lambda, q^{-1}\}\log q$.

  By Prop.~\ref{prop:bi-high-snr-C-bound}, $C^H(q,\lambda) \le q^{5/2}$.
  For $d\lambda^2>1$, we have
  \begin{align}
    c^H(q,\lambda,d)\ge \left(\frac 2q + \frac {q-2}q \cdot \max\{\lambda,0\}\right)^{-1} \ge \left(\frac 2q + \max\{\lambda,0\}\right)^{-1} \ge \frac 14 \max\{\lambda,q^{-1}\}^{-1}.
  \end{align}
  Therefore
  \begin{align}
    d\lambda^2 \exp(-c^H(q,\lambda,d) \cdot \frac{d\lambda^2-1}2) C^H(q,\lambda)
    \le d\lambda^2 \exp\left(-\frac{d\lambda^2-1}{8 \max\{\lambda,q^{-1}\}}\right) q^{5/2} =: g_{q,\lambda}(d).
  \end{align}
  Computing $g'_{q,\lambda}(d)$, we see that $g_{q,\lambda}(d)$ is monotone decreasing in $d$ when $d\lambda^2 > 8\max\{\lambda,q^{-1}\}$.
  Therefore it suffices to prove $g_{q,\lambda}(d_0)<1$ where $d_0\lambda^2 = 1+56\max\{\lambda,q^{-1}\}\log q$.
  We have
  \begin{align}
    g_{q,\lambda}(d_0) = (1+56 \max\{\lambda,q^{-1}\}\log q) \exp(-7 \log q) q^{5/2} \le (1+56 \log q) q^{-9/2}.
  \end{align}
  The last expression is $<1$ for all $q\ge 3$. This finishes the proof.
\end{proof}

\subsection{The degradation method}
% Let $M_k$ denote the FMS channel $\sigma_\rho \to (\sigma_{L_k}, \omega_{T_k})$, and $\wt M_k$ denote the FMS channel $\sigma_\rho \to \omega_{T_k}$ (both with the obvious symmetries). Then $\wt M_k \le_{\deg} M_k$ by forgetting $\sigma_{L_k}$.
% Both sequences $\{M_k\}_{k\ge 0}$ and $\{\wt M_k\}_{k\ge 0}$ satisfy the Belief Propagation recursion, i.e.,
% \begin{gather}
%   M_{k+1} = \BP(M_k), \qquad \wt M_{k+1} = \BP(\wt M_k),\\
%   \BP(M) := \bE_b [(M_k \circ P_\lambda)^{\star b}\star W],
% \end{gather}
% where $b$ follows the branching number distribution (constant if working with regular trees), and $W$ is the survey FMS channel. They start with different initial values $M_0, \wt M_0$ (both FMS channels) where $\wt M_0 \le_{\deg} M_0$.
% We will make one of two assumptions: (1) survey FMS channel is non-trivial, or (2) $\wt M_0$ is non-trivial. The first assumption corresponds to boundary irrelevance, and the second assumption corresponds to uniqueness of BP fixed point.
Let $(M_k)_{k\ge 0}$ and $(\wt M_k)_{k\ge 0}$ be two sequences of $q$-FMS channels satisfying the belief propagation recursion, i.e.,
\begin{gather}
  M_{k+1} = \BP(M_k), \qquad \wt M_{k+1} = \BP(\wt M_k),\\
  \BP(M) := \bE_b [(M_k \circ P_\lambda)^{\star b}\star W],
\end{gather}
where $b$ follows the branching number distribution (constant if working with regular trees), and $W$ is the survey FMS channel (trivial if there is no survey).

For the boundary irrelevance problem, we take $\wt M_0 = 0$, $M_0 = \Id$.
For stability of BP fixed point, we take $M_0=\Id$ and $\wt M_0$ be a given non-trivial FMS channel.
Our goal is to show the the limit $\lim_{k\to \infty} M_k$ and $\lim_{k\to \infty} \wt M_k$ both exist in the sense of weak convergence (Section~\ref{sec:chan-lim}) and are equal.

From now on, we assume that $M_0=\Id$, and either (1) $W$ is non-trivial and $\wt M_0=0$, or (2) $W=0$ and $\wt M_0$ is non-trivial.
Note that in both cases, the initial channels satisfy $\wt M_0 \le_{\deg} M_0$.
Because the BP operator preserves degradation preorder, we have $\wt M_k \le_{\deg} M_k$ for all $k\ge 0$.
So the two channel sequences are naturally related to each other by degradation.

Because $M_0=\Id$, we have $M_k \ge_{\deg} M_{k+1}$ for all $k\ge 0$.
Therefore by Lemma~\ref{lem:deg-down-seq-limit}, $M_\infty := \lim_{k\to \infty} M_k$ exists.
For the boundary irrelevance problem, we also have $M_k \le_{\deg} M_{k+1}$ for all $k\ge 0$, and by Lemma~\ref{lem:deg-up-seq-limit}, $\wt M_\infty := \lim_{k\to \infty} \wt M_k$ exists.
However, for the stability of BP fixed point problem, it is a priori unclear whether the limit $\lim_{k\to \infty} \wt M_k$ exists.

We prove the limit $\lim_{k\to \infty} \wt M_k$ exists and is equal to $M_\infty$ by generalizing the degradation method from \cite{abbe2021stochastic}.
Let $\phi: \cP([q]) \to \bR$ be a strongly convex function invariant under $\Aut([q])$ action.
Extend it to a function $\Phi: \{\text{FMS channels}\} \to \bR$ as $\Phi(P) = \bE \phi(\pi_P)$.
By degradation, we have $\Phi(M_k) \ge \Phi(\wt M_k)$ for all $k\ge 0$.
The following proposition shows that it suffices to prove contraction of potential function $\Phi$.
\begin{proposition} \label{prop:phi-contraction-imply-l2}
  Assume that $\phi: \cP([q]) \to \bR$ is $\alpha$-strongly convex for some $\alpha>0$, and that
  \begin{align}
    \lim_{k\to \infty} (\Phi(M_k)-\Phi(\wt M_k))=0.
  \end{align}
  Then under the canonical coupling, we have
  \begin{align}
    \lim_{k\to \infty} \bE \|\pi_k-\wt \pi_k\|_2^2 =0,
  \end{align}
  where $\pi_k$ (resp.~$\wt \pi_k$) is the $\pi$-component of $M_k$ (resp.~$\wt M_k$).
  In particular, if $M_0=\Id$, then both limits $\lim_{k\to \infty} M_k$ and $\lim_{k\to \infty} \wt M_k$ exist in the sense of weak convergence, and the two limits are equal.
\end{proposition}
\begin{proof}
  \begin{align}
    \Phi(M_k) - \Phi(\wt M_k) &= \bE_{\wt \pi_k} \bE[ \phi(\pi_k) - \phi(\wt \pi_k) | \wt \pi_k] \\
    &\ge \bE_{\wt \pi_k} \bE[\langle \nabla \phi(\wt \pi_k), \pi_k - \wt \pi_k\rangle + \frac \alpha 2 \|\pi_k - \wt \pi_k\|_2^2 | \wt \pi_k] \nonumber \\
    &\ge \bE_{\wt \pi_k} \langle \nabla \phi(\wt \pi_k), \bE[\pi_k | \wt \pi_k]- \wt \pi_k\rangle + \frac \alpha 2 \bE \|\pi_k - \wt \pi_k\|_2^2 \nonumber \\
    & \ge \frac \alpha 2 \bE \|\pi_k - \wt \pi_k\|_2^2, \nonumber
  \end{align}
  where the second step is by $\alpha$-strongly convexity,
  and the third step is because $\wt \pi_k\le_m \bE[\pi_k | \wt \pi_k]$ and $\phi$ is convex (thus Schur-convex).
  Taking the limit $k\to \infty$, we see that the Wasserstein $W_2$ distance between the $\pi$-distributions of $M_k$ and $\wt M_k$ goes to $0$.
  Because $\lim_{k\to \infty} M_k$ converges weakly to a limit $M_\infty$, $\lim_{k\to \infty} \wt M_k$ also converges to the same limit.
\end{proof}
\begin{proposition} \label{prop:contraction-imply-bi}
  Assume that $\phi: \cP([q]) \to \bR$ is $\alpha$-strongly convex for some $\alpha>0$, and that
  \begin{align}
    \lim_{k\to \infty} (\Phi(M_k)-\Phi(\wt M_k))=0.
  \end{align}
  whenever $M_0=\Id$ and (1) $W$ is non-trivial and $\wt M_0=0$, or (2) $W=0$ and $\wt M_0$ is non-trivial.
  Then boundary irrelevance and stability of BP fixed point hold.
\end{proposition}
\begin{proof}
  \textbf{Boundary irrelevance:}
  Let $\wt M_0 = 0$, $M_0 = \Id$.
  By Prop.~\ref{prop:phi-contraction-imply-l2}, we have
  \begin{align}
    \lim_{k\to \infty} M_k = \lim_{k\to \infty} \wt M_k.
  \end{align}
  In particular,
  \begin{align}
    \lim_{k\to \infty} C(M_k) = \lim_{k\to \infty} C(\wt M_k)
  \end{align}
  where $C$ denote capacity (Definition~\ref{defn:fms-info-measure}).
  Note that
  \begin{align}
    \lim_{k\to \infty} C(M_k) &= \lim_{k\to\infty} I(\sigma_\rho; \sigma_{L_k}, \omega_{T_k} |T_k),\\
    \lim_{k\to \infty} C(\wt M_k) &= \lim_{k\to\infty} I(\sigma_\rho; \omega_{T_k} |T_k).
  \end{align}
  So this proves boundary irrelevance.

  \textbf{Stability of BP fixed point:}
  Suppose there is a non-trivial fixed point FMS channel $U$.
  Let $\wt M_0 = U$, $M_0 = \Id$.
  By Prop.~\ref{prop:phi-contraction-imply-l2}, we have
  \begin{align}
    \lim_{k\to \infty} M_k = \lim_{k\to \infty} \wt M_k.
  \end{align}
  Because $U$ is a fixed point, LHS is equal to $U$.
  On the other hand, RHS does not depend on $U$.
  Therefore there is a unique non-trivial FMS fixed point.
\end{proof}

% By strictly convexity, if
% \begin{align}
% 	\lim_{k\to \infty} \Phi(M_k) = \lim_{k\to \infty} \Phi(\wt M_k),
% \end{align}
% then $\{M_k\}_{k\ge 0}$ and $\{\wt M_k\}_{k\ge 0}$ converge in distribution to the same FMS channel.
% Therefore it suffices to prove contraction of $\Phi(M_k) - \Phi(\wt M_k)$ under BP recursion.

We remark that it might be possible to extend the degradation method to asymmetric models. See Section~\ref{sec:discuss} for more discussions.

\subsection{Low SNR} \label{sec:low-snr}
% \yuzhou{Check original uniqueness proof}
For the low SNR case, we use SKL capacity as the potential function. We define
\begin{align}
	\phi^L(\pi) = C_{\SKL}(\FSC_\pi) = \sum_{i\in [q]} (\pi_i - \frac 1q) \log \pi_i.
\end{align}
It is useful for our purpose because it is strongly convex (Lemma \ref{lem:phi-L-convex}) and additive under $\star$-convolution (Lemma \ref{lem:phi-L-additive-under-star}). Using this properties we show that desired contraction holds (Prop.~\ref{prop:phi-L-contraction}). We defer the proof of Theorem~\ref{thm:bi-low-snr} to Section~\ref{sec:proof-bi:low-snr}.

\subsection{High SNR} \label{sec:high-snr}
For the high SNR case, we use Bhattacharyya coefficient as the potential function.
We define
\begin{align}
	\phi^H(\pi) = Z(\FSC_\pi^R) = \frac 1{q-1} \left(\left(\sum_{i\in [q]} \sqrt {\pi_i}\right)^2 -1\right).
\end{align}
It is useful for our purpose because it is strongly concave (Lemma \ref{lem:phi-H-concave}) and multiplicative under $\star$-convolution (Lemma \ref{lem:phi-H-mult-under-star}). Using this properties we show that desired contraction holds (Prop.~\ref{prop:phi-H-contraction}). We defer the proof of Theorem~\ref{thm:bi-high-snr} to Section~\ref{sec:proof-bi:high-snr}.

\section{Boundary irrelevance does not always hold} \label{sec:non-bi}
In this section we prove that boundary irrelevance does not hold for the Potts model between the reconstruction threshold and the Kesten-Stigum threshold.
\begin{proposition} \label{prop:no-robust-recon}
  In the setting of Theorem~\ref{thm:non-bi}, for all $\delta>0$, there exists $\epsilon>0$ such that for any FMS survey channel $W$ with $C_{\chi^2}(W) \le \epsilon$, we have
  \begin{align}
    \lim_{k\to \infty} I_{\chi^2}(\sigma_\rho; \omega_{T_k} | T_k) \le \delta.
  \end{align}
\end{proposition}

\begin{proof}[Proof of Theorem~\ref{thm:non-bi} given Prop.~\ref{prop:no-robust-recon}]
  In the reconstruction regime,
  \begin{align}
    \lim_{k\to \infty} I(\sigma_\rho; \sigma_{L_k}, \omega_{T_k} | T_k) \ge \lim_{k\to \infty} I(\sigma_\rho; \sigma_{L_k} | T_k) > 0.
  \end{align}

  Take $\delta>0$ such that $\delta \log 2 < \lim_{k\to \infty} I(\sigma_\rho; \sigma_{L_k} | T_k)$.
  Because $I\le I_{\chi^2} \log 2$, and by Prop.~\ref{prop:no-robust-recon}, for weak enough survey channel $W$ we have
  \begin{align}
    \lim_{k\to \infty} I(\sigma_\rho; \omega_{T_k} | T_k) \le \lim_{k\to \infty} I_{\chi^2}(\sigma_\rho; \omega_{T_k} | T_k) \log 2 \le \delta \log 2 < \lim_{k\to \infty} I(\sigma_\rho; \sigma_{L_k}, \omega_{T_k} | T_k).
  \end{align}
  Therefore
  \begin{align}
    \lim_{k\to \infty} I(\sigma_\rho; \sigma_{L_k} | T_k, \omega_{T_k})
    &= \lim_{k\to \infty} (I(\sigma_\rho; \sigma_{L_k}, \omega_{T_k} | T_k) - I(\sigma_\rho; \omega_{T_k} | T_k)) \\
    &= \lim_{k\to \infty} I(\sigma_\rho; \sigma_{L_k}, \omega_{T_k} | T_k) - \lim_{k\to \infty} I(\sigma_\rho; \omega_{T_k} | T_k)
    > 0. \nonumber
  \end{align}
\end{proof}

Proof of Prop.~\ref{prop:no-robust-recon} is deferred to Section~\ref{sec:proof-non-bi:no-robust-recon}.
The proof uses contraction and local subadditivity properties of the information measure $C_{\chi^2}$.

\begin{lemma}[Contraction] \label{lem:contraction}
  For any FMS channel $P$, we have
  \begin{align}
    C_{\chi^2}(P\circ P_\lambda) \le \lambda^2 C_{\chi^2}(P).
  \end{align}
\end{lemma}
\begin{proof}
  By reversibility and $\chi^2$-contraction coefficient: $\eta_{\chi^2}(\Unif([q]), P_\lambda) = \lambda^2$.
\end{proof}

\begin{lemma}[Local subadditivity] \label{lem:subadd}
  Fix $q\in \bZ_{\ge 2}$. For any $\epsilon>0$ and $q$-FMS channels $P$, $Q$ with $C_{\chi^2}(P) \le \epsilon$, we have
  \begin{align}
    C_{\chi^2}(P\star Q) \le (1+O_q(\epsilon^{1/5})) (C_{\chi^2}(P) + C_{\chi^2}(Q)),
  \end{align}
  where $O_q$ hides a constant depending on $q$.
  % There exists a constant $c=c(q)>0$ such that for any $\delta>0$ and $0<\epsilon \le c \delta^6$, the following holds: if $P$ and $Q$ are two FMS channels with $C_{\chi^2}(P) \le \epsilon$, $C_{\chi^2}(Q) \le \epsilon$, then
  % \begin{align}
  %   C_{\chi^2}(P\star Q) \le (1+\delta) (C_{\chi^2}(P) + C_{\chi^2}(Q)).
  % \end{align}
\end{lemma}
Proof is deferred to Section \ref{sec:proof-non-bi:subadd}.

\ifdefined\isarxiv
\section*{Acknowledgments}
We thank Emmanuel Abbe, Elisabetta Cornacchia, Jingbo Liu, Youngtak Sohn for helpful discussions.
We thank the anonymous reviewers for helpful comments.

Research was sponsored by the United States Air Force Research Laboratory and the United States Air Force Artificial Intelligence Accelerator and was accomplished under Cooperative Agreement Number FA8750-19-2-1000. The views and conclusions contained in this document are those of the authors and should not be interpreted as representing the official policies, either expressed or implied, of the United States Air Force or the U.S. Government. The U.S. Government is authorized to reproduce and distribute reprints for Government purposes notwithstanding any copyright notation herein.
\else
\acks{We thank Emmanuel Abbe, Elisabetta Cornacchia, Jingbo Liu, Youngtak Sohn for helpful discussions.
We thank the anonymous reviewers for helpful comments.

Research was sponsored by the United States Air Force Research Laboratory and the United States Air Force Artificial Intelligence Accelerator and was accomplished under Cooperative Agreement Number FA8750-19-2-1000. The views and conclusions contained in this document are those of the authors and should not be interpreted as representing the official policies, either expressed or implied, of the United States Air Force or the U.S. Government. The U.S. Government is authorized to reproduce and distribute reprints for Government purposes notwithstanding any copyright notation herein.}
\fi

\ifdefined\isarxiv
\bibliographystyle{alpha}
\fi
\bibliography{ref}

\appendix

\section{Limit of channels} \label{sec:chan-lim}
In this section we build the foundation for discussing limits of information channels.
We view a channel $P:\cX\to \cY$ as a distribution of posterior distributions under uniform prior, i.e., the distribution of $P_{X|Y}$ where $P_X=\Unif(\cX)$, $Y\sim P(\cdot | X)$.
Let $\mu$ denote the posterior distribution variable and $P_\mu\in \cP(\cP(\cX))$ be its distribution (called $P$'s posterior distribution's distrbution).
Note that $P_\mu$ is invariant under channel equivalence.

We often work with sequences $(P_k)_{k\ge 0}$ of channels with the same input alphabet $\cX$.
Let $P_{\pi,k}$ denote the distrbution of posterior distributions of $P_k$ under uniform prior.
Let $P_\infty$ be a channel with input alphabet $\cX$ and posterior distribution's distrbution $P_{\pi,\infty}$.
We say $(P_k)_{k\ge 0}$ converges weakly to $P_\infty$ if $(P_{\pi,k})_{k\ge 0}$ converges weakly to $P_{\pi,\infty}$ as distributions on $\cP(\cX)$.

In general, given such a sequence, a limit does not necessarily exist. Nevertheless, when the channels are related to each other via degradation, a limit channel exists.
\begin{lemma} \label{lem:deg-down-seq-limit}
  Let $(P_k : \cX \to \cY_k)_{k\ge 0}$ be a sequence of channels with the same finite input alphabet.
  If $P_k\ge_{\deg} P_{k+1}$ for all $k$, then $(P_k)_{k\ge 0}$ converges weakly to some channel $P_\infty$.
\end{lemma}
\begin{proof}
  By definition of degradation, there exists channel $R_k: \cY_k \to \cY_{k+1}$ such that $P_{k+1} = R_k \circ P_k$.
  This gives rise to an infinite Markov chain
  \begin{align}
    X - Y_0 - Y_1 - Y_2 - \cdots.
  \end{align}
  Let $\mu_k$ denote the posterior distribution variable $P_{X|Y_k}$.
  Then we have
  \begin{align}
    \bE[\mu_{k-1} | Y_k] = \mu_k.
  \end{align}
  Let $\cF_k$ denote the $\sigma$-algebra generated by $(Y_i)_{i\ge k}$.
  Then $(\cF_k)_{k\ge 0}$ is a reverse filtration and $(\mu_k)_{k\ge 0}$ is a reverse martingale with respect to $(\cF_k)_{k\ge 0}$.
  By reverse martingale convergence theorem (e.g.,~\cite[Theorem 4.7.1]{durrett2019probability}),
  $\lim_{k\to \infty} \mu_k$ converges almost surely.
  Define $\mu_\infty := \lim_{k\to \infty} \mu_k$.
  Let $P_\infty$ be a channel with input alphabet $\cX$ whose posterior distribution's distribution is $\mu_\infty$.
  Then $(P_k)_{k\ge 0}$ converges weakly to $P_\infty$.
\end{proof}

\begin{lemma} \label{lem:deg-up-seq-limit}
  Let $(P_k : \cX \to \cY_k)_{k\ge 0}$ be a sequence of channels with the same finite input alphabet.
  If $P_k\le_{\deg} P_{k+1}$ for all $k$, then $(P_k)_{k\ge 0}$ converges weakly to some channel $P_\infty$.
\end{lemma}
\begin{proof}
  By definition of degradation, there exists channel $R_k: \cY_k \to \cY_{k-1}$ such that $P_{k-1} = R_k \circ P_k$.
  This gives rise to an infinite Markov chain
  \begin{align}
    X - Y_0 - Y_1 - Y_2 - \cdots.
  \end{align}
  Let $\mu_k$ denote the posterior distribution variable $P_{X|Y_k}$.
  Then we have
  \begin{align}
    \bE[\mu_{k+1} | Y_k] = \mu_k.
  \end{align}
  Let $\cF_k$ denote the $\sigma$-algebra generated by $(Y_i)_{i\le k}$.
  Then $(\cF_k)_{k\ge 0}$ is a filtration and $(\mu_k)_{k\ge 0}$ is a martingale with respect to $(\cF_k)_{k\ge 0}$.
  Note that the variables $\mu_k$ take values in $\cP(\cX)$, so are uniformly bounded.
  By martingale convergence theorem (e.g.,~\cite[Theorem 4.2.11]{durrett2019probability}),
  $\lim_{k\to \infty} \mu_k$ converges almost surely.
  Define $\mu_\infty := \lim_{k\to \infty} \mu_k$.
  Let $P_\infty$ be a channel with input alphabet $\cX$ whose posterior distribution's distribution is $\mu_\infty$.
  Then $(P_k)_{k\ge 0}$ converges weakly to $P_\infty$.
\end{proof}

By symmetry, in Lemma~\ref{lem:deg-down-seq-limit} and Lemma~\ref{lem:deg-up-seq-limit}, if the sequence $(P_k)_{k\ge 0}$ consists of FMS channels, then the limit $P_\infty$ is an FMS channel.

\section{Proofs in Section~\ref{sec:fms}} \label{sec:proof-fms}
\begin{proof}[Proof of Prop.~\ref{prop:fms-mixture}]
  \textbf{Existence:}
  The proof strategy is to partition $\cY$ into $\Aut(\cX)$-orbits and show that the channel $P$ restricted to each orbit is equivalent to an FSC.

  \textbf{Step 1.} We first prove that we can replace $P$ with an equivalent FMS channel whose $\Aut(\cX)$ action is free, so that in later steps each orbit is easier to handle.
  Define channel $\wt P: \cX \to \cY \times \wt \cY$,
  where $\wt \cY = \Aut(\cX)$, sending $X$ to $(Y, \wt Y)$ where $Y\sim P(\cdot | X)$ and $\wt Y \sim \Unif(\Aut(\cX))$ is independent of $X$.
  We give $\wt P$ an FMS structure where $\Aut(\cX)$ acts on $\wt \cY$ by left multiplication.
  It is easy to see that $P$ is equivalent to $\wt P$.
  Therefore we can replace $P$ with $\wt P$ and wlog assume that $\Aut(\cX)$ action is free.

  \textbf{Step 2.}
  Let $\cO = \cY / \Aut(\cX)$ be the space of orbits of the $\Aut(\cX)$ action on $\cY$.
  For an orbit $o\in \cO$, for any two elements $y_1,y_2\in o$, the posterior distributions $\pi_1 = P_{X|Y=y_1}$ and $\pi_2 = P_{X|Y=y_2}$ (with uniform priors) differ by a permutation, by the assumption that $P$ is FMS.
  In particular, $\pi_1$ and $\pi_2$ map to the same element in $\cP(\cX)/\Aut(\cX)$.
  Therefore we can uniquely assign an element $\pi_o\in \cP(\cX)/\Aut(\cX)$ for any $o\in \cO$.

  Note that by symmetry, the distribution of $o$ does not depend on the input distribution. Let $P_o \in \cP(\cO)$ be this distribution.
  Then $P$ is equivalent to the channel $X\to (o, Z)$ where $o\sim P_o$ is independent of $X$, and $Z\sim \FSC_{\pi_o}(\cdot | X)$.
  (Because $\Aut(\cX)$ action on $\cY$ is free, this equivalence is in fact just renaming the output space.)

  \textbf{Step 3.}
  Finally we prove that the FMS channel $X\to (o, Z)$ is equivalent to $X\to (\pi_o,Z)$.
  One side is easy: given $(o,Z)$, we can generate $(\pi_o,Z)$.
  For the other side, given $(\pi,Z)$, we can generate $o' \sim P_{o | \pi_o = \pi}$.
  Then $(o',Z)$ has the same distribution as $(o,Z)$, conditioned on any input distribution.
  This finishes the existence proof.

  \textbf{Uniqueness:}
  For any FMS channel $X\xrightarrow{Q}Y$, we can associate it with a distribution $Q_\pi$ on $\cP(\cX)/\Aut(\cX)$, defined as the distribution of the posterior distribution $Q_{X|Y}$, where $Y\sim Q_{Y|X} \circ \Unif(\cX)$ is generated with uniform prior distribution. (By definition $Q_\pi$ is a distribution on $\cP(\cX)$. However, by symmetry property of FMS, $Q_\pi$ is invariant under $\Aut(\cX)$ action.)
  It is easy to see that $Q_\pi$ distribution is preserved under equivalence between FMS channels.
  Furthermore, for an FMS channel of form $X\to (\pi,Z)$ as described in the proposition statement, this distribution of posterior distribution is equal to $P_\pi$.
  Therefore $P_\pi$ is uniquely deteremined by $P$.
\end{proof}

\begin{proof}[Proof of Prop.~\ref{prop:fms-deg-couple}]
  \textbf{Degradation $\Rightarrow$ Coupling:}
  Say $P$ maps $X$ to $Y$, and $Q$ maps $X$ to $Z$.
  Let $\pi'_P \in \cP(\cX)$ be the posterior distribution of input $X$ given output $Y$, where $Y\sim P_{Y|X} \circ \Unif(\cX)$ is generated with uniform prior distribution.
  Similarly define $\pi'_Q$.
  Then $\pi_P$ (resp.~$\pi_Q$) is the orbit of $\pi'_P$ (resp.~$\pi'_Q$) under permutation.

  Degradation relationship $P = R\circ Q$ induces a coupling on the posterior distributions $\pi'_P$ and $\pi'_Q$. One can check that this coupling is invariant under $\Aut(\cX)$ action and satisfies
  \begin{align}
    \pi' = \bE[\pi'_Q | \pi'_P = \pi'] \qquad \forall \pi' \in \cP(\cX).
  \end{align}
  For any $\pi' \in \cP(\cX)$, let $p(\pi')\in \cP(\cX)/\Aut(\cX)$ denotes its projection.
  Then we have
  \begin{align}
    p(\pi') \le_m \bE[p(\pi'_Q) | \pi'_P = \pi'].
  \end{align}
  Taking expectation over the orbit, we get
  \begin{align}
    \pi \le_m \bE[\pi_Q | \pi_P = \pi] \qquad \forall \pi \in \cP(\cX)/\Aut(\cX).
  \end{align}

  \textbf{Coupling $\Rightarrow$ Degradation:}
  \textbf{Step 1.} We prove that for $\pi,\pi'\in \cP(\cX)/\Aut(\cX)$,
  if $\pi \le_m \pi'$, then $\FSC_\pi \le_{\deg} \FSC_{\pi'}$.
  Because $\pi\le_m \pi'$, there exists $a\in \cP(\Aut(\cX))$ such that (see e.g., \cite[2.20]{hardy1934inequalities})
  \begin{align}
    \pi_i = \sum_{\sigma \in \Aut(\cX)} a_\sigma \pi'_{\sigma^{-1}(i)} \qquad \forall i\in \cX.
  \end{align}

  For $\rho\in \Aut(\cX)$, we have
  \begin{align}
    \FSC_\pi(\rho | i) = \frac 1{(q-1)!} \pi_{\rho^{-1}(i)}
    = \sum_{\sigma\in \Aut(\cX)} a_\sigma \frac 1{(q-1)!} \pi'_{\sigma^{-1}\rho^{-1}(i)}
    = \sum_{\sigma\in \Aut(\cX)} a_\sigma \FSC_{\pi'}(\rho \sigma | i).
  \end{align}
  Therefore we can let $R$ map $\rho\sigma$ to $\rho$ with probability $a_\sigma$, for all $\sigma\in \Aut(\cX)$.
  This gives the desired degradation map $R$.

  \textbf{Step 2.}
  We use the FSC mixture representation (Prop.~\ref{prop:fms-mixture}).
  Suppose $P$ maps $X$ to $(\pi_P, Z_P)$, and $Q$ maps $X$ to $(\pi_Q, Z_Q)$.
  If
  \begin{align}
    \pi = \bE[\pi_Q | \pi_P = \pi] \qquad \forall \pi \in \cP(\cX)/\Aut(\cX),
  \end{align}
  then we can construct $R$ by mapping $\pi_Q$ to coupled $\pi_P$ (randomly), and keeping the $Z$ component.

  Now define an FMS channel $\wt P$ whose $\pi$-component is $f(\pi_P)$, where
  \begin{align}
    f(\pi) := \bE[\pi_Q | \pi_P = \pi].
  \end{align}
  Then by Step 1, $P\le_{\deg} \wt P$. By Step 2, $\wt P\le_{\deg} Q$.
  Therefore $P\le_{\deg} Q$.
\end{proof}

\section{Proofs in Section~\ref{sec:bi}} \label{sec:proof-bi}

\subsection{Proofs for low SNR case} \label{sec:proof-bi:low-snr}

We state a few properties of the function $\phi^L$.
\begin{lemma} \label{lem:phi-L-convex}
	$\phi^L$ is $1$-strongly convex on $\cP([q])$.
\end{lemma}
\begin{proof}
	\begin{align}
		\nabla^2 \phi^L(\pi) = \diag\left(\pi^{-1} + \frac 1q \pi^{-2}\right) \succeq I.
	\end{align}
\end{proof}

\begin{lemma} \label{lem:phi-L-additive-under-star}
	$\Phi^L(\cdot) = C_{\SKL}(\cdot)$ is additive under $\star$-convolution.
\end{lemma}
\begin{proof}
  The statement follows from additivity of SKL divergence under $\star$-convolution \cite{kulske2009symmetric}.
  For completeness, we present a direct proof using~\eqref{eqn:star-formula}.

  By FSC mixture decomposition (Prop.~\ref{prop:fms-mixture}), it suffices to prove that
  \begin{align}
    \Phi^L(\FSC_\pi \star \FSC_{\pi'}) = \Phi^L(\FSC_\pi) + \Phi^L(\FSC_{\pi'}).
  \end{align}
  We have
  \begin{align*}
    &~\Phi^L(\FSC_\pi \star \FSC_{\pi'}) \\
    =&~ \sum_{\tau \in \Aut([q])} \left(\frac 1{(q-1)!}\sum_{i\in [q]} \pi_i \pi'_{\tau(i)}\right) \phi^L(\pi \star_\tau \pi') \\
    =&~ \sum_{\tau \in \Aut([q])} \left(\frac 1{(q-1)!}\sum_{i\in [q]} \pi_i \pi'_{\tau(i)}\right)
    \sum_{j\in [q]} \left((\pi \star_\tau \pi')_j - \frac 1q\right)\log (\pi \star_\tau \pi')_j\\
    =&~ \sum_{\tau \in \Aut([q])} \left(\frac 1{(q-1)!}\sum_{i\in [q]} \pi_i \pi'_{\tau(i)}\right)
    \sum_{j\in [q]} \left(\frac {\pi_j \pi'_{\tau(j)}}{\sum_{k\in [q]} \pi_k \pi'_{\tau(k)}} - \frac 1q\right)\log \frac {\pi_j \pi'_{\tau(j)}}{\sum_{k\in [q]} \pi_k \pi'_{\tau(k)}}\\
    =&~ \sum_{\tau \in \Aut([q])} \frac 1{(q-1)!}
    \sum_{j\in [q]} \left(\pi_j \pi'_{\tau(j)} - \frac 1q \sum_{i\in [q]} \pi_i \pi'_{\tau(i)} \right)\log \frac {\pi_j \pi'_{\tau(j)}}{\sum_{k\in [q]} \pi_k \pi'_{\tau(k)}} \\
    =&~ \sum_{\tau \in \Aut([q])} \frac 1{(q-1)!}
    \sum_{j\in [q]} \left(\pi_j \pi'_{\tau(j)} - \frac 1q \sum_{i\in [q]} \pi_i \pi'_{\tau(i)} \right)\log (\pi_j \pi'_{\tau(j)}) \\
    =&~ \sum_{j\in [q]} (\pi_j-\frac 1q)\log \pi_j + \sum_{j\in [q]} (\pi'_j-\frac 1q)\log \pi'_j \\
    =&~ \Phi^L(\FSC_\pi) + \Phi^L(\FSC_{\pi'}).
  \end{align*}
\end{proof}

Condition~\eqref{eqn:thm:bi-low-snr-cond} implies the desired contraction.
\begin{proposition} \label{prop:phi-L-contraction}
	If \eqref{eqn:thm:bi-low-snr-cond} holds, then
	$
		\lim_{k\to \infty} \left(\Phi^L(M_{k}) - \Phi^L(\wt M_{k})\right) = 0.
	$
\end{proposition}
\begin{proof}
  Using BP equation and Lemma~\ref{lem:phi-L-additive-under-star}, we get
	\begin{align}
		\Phi^L(M_{k+1}) = \bE_b \left[b \Phi^L(M_k\circ P_\lambda) + \Phi^L(W)\right] = d \Phi^L(M_k\circ P_\lambda) + \Phi^L(W),
	\end{align}
	and the same holds with $M$ replaced with $\wt M$.

	To prove that
	\begin{align}
		\Phi^L(M_{k+1}) - \Phi^L(\wt M_{k+1}) \le c \left(\Phi^L(M_{k}) - \Phi^L(\wt M_{k})\right),
	\end{align}
	for some $c<1$, it suffices to prove that
	\begin{align}
		d \Phi^L(\FSC_\pi \circ P_\lambda) - c \Phi^L(\FSC_\pi)
		= d \phi^L\left(\lambda \pi + \frac{1-\lambda}q\right) - c \phi^L(\pi)
	\end{align}
	is concave in $\pi$.

  Let $c = d \lambda^2 C^L(q,\lambda)$.
  Then for all $v\in \mathbbm{1}^\perp \in \bR^q$, we have
	\begin{align}
		v^\top \nabla^2 \left(d \phi^L\left(\lambda \pi + \frac{1-\lambda}q\right) - c \phi^L(\pi)\right) v
    = d \lambda^2 f^L\left(\lambda \pi + \frac{1-\lambda}q, v\right) - c f^L(\pi, v)
    \le 0
	\end{align}
  where the first step is because $v^\top \nabla^2 \phi^L(\pi) v = f^L(\pi, v)$ and
  the second step is by definition of $C^L$.
	Therefore contraction holds.
\end{proof}

Theorem~\ref{thm:bi-low-snr} follows from combining everything.
\begin{proof}[Proof of Theorem~\ref{thm:bi-low-snr}]
  By combining Prop.~\ref{prop:phi-L-contraction}, Lemma \ref{lem:phi-L-convex}, and Prop.~\ref{prop:contraction-imply-bi}.
\end{proof}

\subsection{Proofs for high SNR case} \label{sec:proof-bi:high-snr}

We state a few properties of the function $\phi^H$.
\begin{lemma} \label{lem:phi-H-concave}
	$\phi^H$ is $\alpha$-strongly concave on $\cP([q])$ for some $\alpha>0$.
\end{lemma}

\begin{proof}
	For any $\pi \in \cP([q])$ we have
	\begin{align}
		\nabla^2 \phi^H(\pi) = \frac 1{q-1} \left(\frac 12 \left(\pi^{-1/2}\right) \left(\pi^{-1/2}\right)^\top - \frac 12 \left(\sum_{i\in [q]} \pi^{1/2}_i\right)\diag\left(\pi^{-3/2}\right)\right).
	\end{align}
	So for any $v\in \mathbbm{1}^\perp \subseteq \bR^q$,
  \begin{align}
		v^\top \nabla^2 \phi^H(\pi) v &= \frac 1{2(q-1)} \left(\left\langle \pi^{-1/2}, v\right\rangle^2 - \left(\sum_{i\in [q]} \pi^{1/2}_i\right) \left\langle \pi^{-3/2}, v^2\right\rangle\right).
  \end{align}

	Let us prove that
	\begin{align}
		\frac{\left\langle \pi^{-1/2}, v\right\rangle^2}{\left(\sum_{i\in [q]} \pi^{1/2}_i\right) \left\langle \pi^{-3/2}, v^2\right\rangle} \le 1-\frac{1}{\sqrt q}. \label{eqn:h2-hess-ratio-step-1}
	\end{align}
  Performing change of variable $u = \pi^{-3/4} v$, LHS of \eqref{eqn:h2-hess-ratio-step-1} becomes
	\begin{align}
		\frac{\left\langle \pi^{1/4}, u\right\rangle^2}{\left(\sum_{i\in [q]} \pi^{1/2}_i\right) \|u\|_2^2}.
		\label{eqn:h2-hess-ratio-step-1-cov}
	\end{align}
	We would like to maximize this expression over the hyperplane $\left\langle \pi^{3/4}, u\right\rangle=0$.
	By geometric interpretation, maximum value is achieved at projection of $\pi^{1/4}$ onto the hyperplane, i.e.,
	\begin{align}
		u = \pi^{1/4} - \frac{\left\langle \pi^{1/4}, \pi^{3/4}\right\rangle}{\|\pi^{3/4}\|_2^2} \pi^{3/4} = \pi^{1/4} - \frac{\pi^{3/4}}{\|\pi^{3/4}\|_2^2},
	\end{align}
	at which \eqref{eqn:h2-hess-ratio-step-1-cov} achieves value
	\begin{align}
		&~ \frac{\left( \sum_{i\in [q]} \pi_i^{1/2} - \frac 1{\sum_{i\in [q]} \pi_i^{3/2}}\right)^2}{\left(\sum_{i\in [q]} \pi_i^{1/2}\right) \left( \sum_{i\in [q]} \pi_i^{1/2} - \frac 1{\sum_{i\in [q]} \pi_i^{3/2}}\right)} \\
		=&~ \frac{\sum_{i\in [q]} \pi_i^{1/2} - \frac 1{\sum_{i\in [q]} \pi_i^{3/2}}}{\sum_{i\in [q]} \pi_i^{1/2}} \nonumber \\
		=&~ 1 - \frac 1{\left(\sum_{i\in [q]} \pi_i^{1/2}\right) \left(\sum_{i\in [q]} \pi_i^{3/2}\right)} \nonumber \\
		\le &~ 1-\frac 1{\sqrt q} \nonumber
	\end{align}
	where the last step is because
	\begin{align}
		\sum_{i\in [q]} \pi_i^{1/2} \le \sqrt q, \qquad \sum_{i\in [q]} \pi_i^{3/2} \le 1.
	\end{align}
  This finishes the proof of~\eqref{eqn:h2-hess-ratio-step-1}.

  Therefore
  \begin{align}
    v^\top \nabla^2 \phi^H(\pi) v &= \frac 1{2(q-1)} \left(\left\langle \pi^{-1/2}, v\right\rangle^2 - \left(\sum_{i\in [q]} \pi^{1/2}_i\right) \left\langle \pi^{-3/2}, v^2\right\rangle\right) \\
		&\le  -\frac 1{2(q-1)\sqrt q} \left(\sum_{i\in [q]} \pi^{1/2}_i\right) \left\langle \pi^{-3/2}, v^2\right\rangle\nonumber \\
		&\le -\frac 1{2(q-1)\sqrt q} \|v\|_2^2 \nonumber
	\end{align}
	where the second step is by~\eqref{eqn:h2-hess-ratio-step-1}, and the third step is because $\sum_{i\in [q]} \pi_i^{1/2} \ge 1$ and $\pi^{-3/2} \ge 1$.
\end{proof}

\begin{lemma} \label{lem:phi-H-mult-under-star}
	$\Phi^H(\cdot) = Z(\cdot)$ is multiplicative under $\star$-convolution.
\end{lemma}
\begin{proof}
  The statement follows from tensorization property of Hellinger distance (e.g.,~\cite{polyanskiy2023information}).
  For completeness, we present a direct proof using~\eqref{eqn:star-formula}.

  By FSC mixture decomposition (Prop.~\ref{prop:fms-mixture}), it suffices to prove that
  \begin{align}
    \Phi^H(\FSC_\pi \star \FSC_{\pi'}) = \Phi^H(\FSC_\pi) + \Phi^H(\FSC_{\pi'}).
  \end{align}
  We have
  \begin{align*}
    &~\Phi^H(\FSC_\pi \star \FSC_{\pi'}) \\
    =&~ \sum_{\tau \in \Aut([q])} \left(\frac 1{(q-1)!}\sum_{i\in [q]} \pi_i \pi'_{\tau(i)}\right) \phi^H(\pi \star_\tau \pi') \\
    =&~ \sum_{\tau \in \Aut([q])} \left(\frac 1{(q-1)!}\sum_{i\in [q]} \pi_i \pi'_{\tau(i)}\right)
    \frac 1{q-1} \left(\left(\sum_{j\in [q]} \sqrt{(\pi \star_\tau \pi')_j}\right)^2-1\right)\\
    =&~ \sum_{\tau \in \Aut([q])} \left(\frac 1{(q-1)!}\sum_{i\in [q]} \pi_i \pi'_{\tau(i)}\right)
    \frac 1{q-1} \left(\left(\sum_{j\in [q]} \sqrt{\frac {\pi_j \pi'_{\tau(j)}}{\sum_{k\in [q]} \pi_k \pi'_{\tau(k)}}}\right)^2-1\right)\\
    =&~ \sum_{\tau \in \Aut([q])} \frac 1{(q-1)!}
    \cdot \frac 1{q-1} \left(\left(\sum_{j\in [q]} \sqrt{\pi_j \pi'_{\tau(j)}}\right)^2-\sum_{i\in [q]} \pi_i \pi'_{\tau(i)}\right)\\
    =&~ \sum_{\tau \in \Aut([q])} \frac 1{(q-1)!}
    \cdot \frac 1{q-1} \sum_{j\ne k\in [q]} \sqrt{\pi_j \pi'_{\tau(j)} \pi_k \pi'_{\tau(k)}}\\
    =&~ \frac 1{(q-1)^2} \left(\sum_{j\ne k\in [q]} \sqrt{\pi_j \pi_k}\right)\left(\sum_{j'\ne k'\in [q]}\sqrt{\pi'_{\tau(j')} \pi'_{\tau(k')}}\right)\\
    =&~ \Phi^H(\FSC_\pi) \Phi^H(\FSC_{\pi'}).
  \end{align*}
\end{proof}

\begin{lemma}\label{lem:h2-chi2-bound}
	For any BMS channel $P$, we have
	\begin{align}
		Z(P) \le \sqrt{1-C_{\chi^2}(P)}.
	\end{align}
\end{lemma}
\begin{proof}
	Let $\Delta$ be the $\Delta$-component of $P$. Then
	\begin{align}
		Z(P) = \bE [2\sqrt{\Delta(1-\Delta)}] \le \sqrt{1-\bE[(1-2\Delta)^2]} = \sqrt{1-C_{\chi^2}(P)}.
	\end{align}
	The inequality step is by concavity of $\sqrt{\cdot}$.
\end{proof}

Condition~\eqref{eqn:thm:bi-high-snr-cond} implies the desired contraction.
\begin{proposition} \label{prop:phi-H-contraction}
  If \eqref{eqn:thm:bi-high-snr-cond} or \eqref{eqn:thm:bi-w-high-snr-cond} holds, then
	$
		\lim_{k\to \infty} \left(\Phi^H(M_{k}) - \Phi^H(\wt M_{k})\right) = 0.
  $
\end{proposition}
\begin{proof}
	We treat the regular tree case and the Poisson tree case (almost) uniformly.
  For simplicity, in this proof, we use the following notation.
	Let $\mathbbm{1}_R$ be $1$ if we are working with regular trees, and $0$ otherwise.
	Let $\mathbbm{1}_P$ be $1$ if we are working with Poisson trees, and $0$ otherwise.

  Using BP equation and Lemma \ref{lem:phi-H-mult-under-star}, we have
	\begin{align}
		\Phi^H(M_{k+1}) = \bE_b\left[ \left(\Phi^H(M_k \circ P_\lambda)\right)^b \Phi^H(W)\right]
	\end{align}
	and the same holds with $M$ replaced with $\wt M$.

	For $i\ge 0$, define
	\begin{align}
		\Phi_i = \bE_b\left[ \prod_{j\in [b]} \left(\left(\Phi^H(\wt M_k \circ P_\lambda)\right)^{\mathbbm{1}\{j\le i\}} \left(\Phi^H(M_k \circ P_\lambda)\right)^{\mathbbm{1}\{j > i\}} \right) \Phi^H(W)\right].
	\end{align}

	Fix $i\ge 1$. Let us prove that
	\begin{align}
		\Phi_i - \Phi_{i-1} \le c_i \left( \Phi(\wt M_k) - \Phi(M_k)\right)
	\end{align}
	for some constant $c_i$ to be determined later.

	Note that
	\begin{align} \label{eqn:prop-phi-H-contraction-master}
		\Phi_i - \Phi_{i-1}
		=&~ \left(\Phi^H(\wt M_k \circ P_\lambda) - \Phi^H(M_k \circ P_\lambda) \right)
		\\ &~\cdot \bE_b \left[\mathbbm{1}\{b\ge i\}\left(\Phi^H(\wt M_k \circ P_\lambda)\right)^{i-1} \left(\Phi^H(M_k \circ P_\lambda)\right)^{b-i}\right] \Phi^H(W). \nonumber
	\end{align}

	% Define
	% \begin{align}
	% 	f(\pi, v) = -v^\top \nabla^2 \phi^H(\pi) v \label{eqn:h2-hess-form}
	% \end{align}
  % (which is non-negative by Lemma \ref{lem:phi-H-concave})
	% and
	% \begin{align}
	% 	c = \sup_{\substack{\pi \in \cP([q]) \\ v \in \mathbbm{1}^\perp \subseteq \bR^q}} \frac{f\left(\lambda \pi + \frac {1-\lambda} q, v\right)}{f(\pi, v)}. \label{eqn:h2-hess-ratio-const}
	% \end{align}

  % By Lemma~\ref{lem:h2-hess-ratio}, we have
  % \begin{align}
  %   c \le q^{5/2}.
  % \end{align}

  Note that $f^H(\pi, v) = -2(q-1)v^\top \nabla^2 \phi^H(\pi) v$.
  Therefore by definition of $C^H(q,\lambda)$, we have
	\begin{align}
		\nabla^2 \left(\Phi^H(\FSC_\pi \circ P_\lambda) - \lambda^2 C^H(q,\lambda) \Phi^H (\FSC_\pi)\right) \succeq 0.
	\end{align}
	So by degradation,
	\begin{align} \label{eqn:prop-phi-H-contraction-term-1}
		\Phi^H(\wt M_k \circ P_\lambda) - \Phi^H(M_k \circ P_\lambda) \le \lambda^2 C^H(q,\lambda) \left(\Phi^H(\wt M_k) - \Phi^H(M_k)\right).
	\end{align}

  By Prop.~\ref{prop:bot-majority} and Lemma \ref{lem:h2-chi2-bound}, for any $\epsilon>0$, for $k$ large enough, we have
  \begin{align} \label{eqn:prop-phi-H-contraction-term-2}
    \Phi^H(M_k \circ P_\lambda)
    \le \left(1-c^H(q,\lambda,d)\cdot\frac{d\lambda^2-1}{d-\mathbbm{1}_R}+\epsilon\right)_+^{1/2}.
  \end{align}

	Let
	\begin{align} \label{eqn:prop-phi-H-contraction-term-c}
		c_i = \lambda^2 \Phi^H(W) C^H(q,\lambda)  \bE_b \left[\mathbbm{1}\{b\ge i\} \left(1-c^H(q,\lambda,d)\cdot\frac{d\lambda^2-1}{d-\mathbbm{1}_R}+\epsilon\right)_+^{\frac{b-1}2} \right].
	\end{align}

  Combining \eqref{eqn:prop-phi-H-contraction-master}\eqref{eqn:prop-phi-H-contraction-term-1}\eqref{eqn:prop-phi-H-contraction-term-2}\eqref{eqn:prop-phi-H-contraction-term-c}, we get
	\begin{align} \label{eqn:prop-phi-H-contraction-step-i}
		\Phi_i - \Phi_{i-1} \le c_i \left(\Phi^H(\wt M_k) - \Phi^H(M_k)\right).
	\end{align}

	Let us compute sum of $c_i$.
  In the regular tree case, we have
  \begin{align}
		\sum_{i\ge 1} c_i
    =&~ d \lambda^2 \Phi^H(W) C^H(q,\lambda) \left(1-c^H(q,\lambda,d)\cdot\frac{d\lambda^2-1}{d-1}+\epsilon\right)_+^{\frac{d-1}2}.
	\end{align}
  For the Poisson tree case, we have
  \begin{align}
		\sum_{i\ge 1} c_i
    =&~ d \lambda^2 \Phi^H(W) C^H(q,\lambda) \exp\left(-d\left(1-\left(1-c^H(q,\lambda,d)\cdot \frac{d\lambda^2-1}{d}+\epsilon\right)_+^{1/2}\right)\right).
	\end{align}
  Note that $\epsilon>0$ can be chosen to be arbitrarily small.
  Therefore in both cases, for any $\epsilon'>0$, for $k$ large enough, we can choose $c_i$ such that \eqref{eqn:prop-phi-H-contraction-step-i} holds and
  \begin{align}
		\sum_{i\ge 1} c_i
    \le &~ d \lambda^2 \Phi^H(W) C^H(q,\lambda) \exp\left(-c^H(q,\lambda,d)\cdot \frac{d\lambda^2-1}{2}+\epsilon'\right).
	\end{align}

	% \begin{align}
	% 	&~\sum_{i\ge 1} c_i\\
  %   \nonumber =&~ d \lambda^2 \Phi^H(W) C^H(q,\lambda) \cdot \left\{\begin{array}{ll}
	% 		\left(1-\frac{d\lambda^2-1}{d-1}\right)^{\frac{d-1}2} & \text{Regular tree case,} \\
	% 		\exp\left(-d\left(1-\sqrt{1-\frac{d\lambda^2-1}{d}}\right)\right) & \text{Poisson tree case,}
	% 	\end{array}\right. \\
  %   \nonumber \le&~ d \lambda^2 \Phi^H(W) C^H(q,\lambda) \exp\left(-\frac{d \lambda^2-1}{2}\right).
	% 	% &\le d \lambda^2 \Phi^H(W) q^{5/2} \exp\left(-\frac{d \lambda^2-1}{2}\right)
	% \end{align}
  % where the last step is by Lemma~\ref{lem:h2-hess-ratio}.

	Then
	\begin{align}
		&\Phi^H(\wt M_{k+1}) - \Phi^H(M_{k+1}) \\
		&\le \left(\sum_{i\ge 1} c_i\right) \left(\Phi^H(\wt M_k) - \Phi^H(M_k)\right) \nonumber\\
		&\le d \lambda^2 \Phi^H(W) C^H(q,\lambda)  \exp\left(-c^H(q,\lambda,d)\cdot \frac{d\lambda^2-1}{2}+\epsilon'\right) \left(\Phi^H(\wt M_k) - \Phi^H(M_k)\right). \nonumber
	\end{align}
  Because \eqref{eqn:thm:bi-high-snr-cond} or \eqref{eqn:thm:bi-w-high-snr-cond} holds, we can choose $\epsilon'>0$ small enough so that
  \begin{align}
    d \lambda^2 \Phi^H(W) C^H(q,\lambda)  \exp\left(-c^H(q,\lambda,d)\cdot \frac{d\lambda^2-1}{2}+\epsilon'\right) < 1.
  \end{align}
  This leads to the desired contraction.
\end{proof}

Theorem~\ref{thm:bi-high-snr} follows from combining everything.
\begin{proof}[Proof of Theorem~\ref{thm:bi-high-snr}]
  Combine Prop.~\ref{prop:phi-H-contraction}, Lemma \ref{lem:phi-H-concave}, and Prop.~\ref{prop:contraction-imply-bi}.
\end{proof}

\subsection{Majority decider} \label{sec:proof-bi:majority}
\begin{proposition}\label{prop:bot-majority}
	Consider the Potts model $\BOT(q,\lambda,d)$ or $\BOT(q,\lambda,\Pois(d))$ with leaf observations through a non-trivial FMS channel $U$.
	Let $M_k^U$ denote the channel $\sigma_\rho \to \nu_{L_k}$ where $\nu_v \sim U(\cdot |\sigma_v)$.
  Assume that $d\lambda^2>1$.
	Then
	\begin{align}
		\lim_{k\to \infty} C_{\chi^2}\left((M_k^U \circ P_\lambda)^R\right) \ge \left\{\begin{array}{ll} c(q,\lambda,d)\cdot \frac{d\lambda^2-1}{d-1} & \text{Regular tree case,} \\ c(q,\lambda,d) \cdot \frac{d\lambda^2-1}{d} & \text{Poisson tree case.}\end{array}\right.
	\end{align}
  where
  \begin{align}
    c(q,\lambda,d) := \left(\frac 2q + \frac{q-2}q \cdot \frac{d\lambda^2-1}{d\lambda-1}\right)^{-1}.
  \end{align}

  Furthermore, $c(q,\lambda,d) \ge 1$ for all $\lambda \in \left[-\frac 1{q-1},1\right]$ and $d\lambda^2>1$.
\end{proposition}

\begin{proof}
  Let $U^*$ be the reverse channel of $U$.
  Then the composition $U^*\circ U$ is a non-trivial ferromagnetic Potts channel.
  So there exists $\eta>0$ such that $P_\eta \le_{\deg} U$.
  By replacing $U$ with $P_\eta$ (and using Lemma~\ref{lem:fms-deg-info-measure}), we can wlog assume that $U=P_\eta$ for some $\eta>0$.

	Fix any embedding $\{\pm\} \subseteq [q]$.
	Let $e\in \bR^q$ denote the vector with $e_+ = 1, e_-=-1, e_i=0$ for $i\not \in \{\pm\}$.
	Let
	\begin{align}
		S_k = \sum_{v\in L_k} e_{\nu_v}.
	\end{align}
	We view $S_k$ as a channel $[q] \to \bZ$.
	By variational characterization of $\chi^2$-divergence, we have
	\begin{align}
		C_{\chi^2}((M_k^U \circ P_\lambda)^R) \ge \frac{\left(\bE^+[S_k \circ P_\lambda]\right)^2}{\bE^+[S_k^2 \circ P_\lambda]}
	\end{align}
	where $\bE^+$ denotes expectation conditioned on root label being $+$.
	Similarly, we use $\bE^-$ to denote expectation conditioned on root label being $-$, and use $\bE^0$ to denote expectation conditioned on root label being any label not $\pm$.
	Same for $\Var^+$, $\Var^-$, $\Var^0$.

	For simplicity, in this proof, we use the following notation.
	Let $\mathbbm{1}_R$ be $1$ if we are working with regular trees, and $0$ otherwise.
	Let $\mathbbm{1}_P$ be $1$ if we are working with Poisson trees, and $0$ otherwise.
  Clearly $\mathbbm{1}_P + \mathbbm{1}_R = 1$.

	It is easy to see that
	\begin{align}
		\bE^i S_k = e_i \eta (d\lambda)^k.
	\end{align}

	Using variance decomposition formula, we have
	\begin{align}
		\Var^i(S_{k+1}) & = \Var^i(\bE[S_{k+1} | b]) + \bE_b \Var^i(\bE[S_{k+1} | b, \sigma_1, \ldots, \sigma_b])\\
		& + \bE \Var^i( S_{k+1} | b, \sigma_1,\ldots,\sigma_b) \nonumber
	\end{align}
	where $\sigma_1,\ldots, \sigma_b$ are labels of the children.

	Let us compute each summand.
	\begin{align}
		&\Var^i(\bE[S_{k+1} | b]) = \Var^i(b \lambda e_i \eta (d\lambda)^k)
		= e_i^2 d \lambda^2 \eta^2 (d\lambda)^{2k} \mathbbm{1}_P,\\
		&\bE_b \Var^i (\bE[S_{k+1} | b, \sigma_1, \ldots, \sigma_b])
		= d \eta^2 (d\lambda)^{2k} \Var_{j\sim P_\lambda(\cdot | i)}(e_j),\\
		&\bE \Var^i( S_{k+1} | b, \sigma_1,\ldots,\sigma_b)
		 = d \bE_{j\sim P_\lambda(\cdot | i)}[\Var^j(S_k)].
	\end{align}

	We have $\Var^-(S_k) = \Var^+(S_k)$ and
	\begin{align}\label{eqn:maj-var-plus-recursion}
		\Var^+(S_{k+1}) =&~ d \eta^2 (d\lambda)^{2k} \left(\left(\lambda + \frac{1-\lambda}q\cdot 2\right) - \lambda^2 \mathbbm{1}_R\right)\\
		&~+ d \left(\left(\lambda + \frac{1-\lambda}q\cdot 2\right) \Var^+(S_k) + \frac{1-\lambda}q\cdot (q-2) \Var^0(S_k)\right), \nonumber\\
    \label{eqn:maj-var-0-recursion}
		\Var^0(S_{k+1}) =&~ d \eta^2 (d\lambda)^{2k} \left(\frac{1-\lambda}q\cdot 2\right)\\
		&~+ d \left(\frac{1-\lambda}q\cdot 2 \Var^+(S_k) + \left(\lambda + \frac{1-\lambda}q\cdot (q-2)\right) \Var^0(S_k)\right). \nonumber
	\end{align}
	By computing linear combinations of \eqref{eqn:maj-var-plus-recursion}\eqref{eqn:maj-var-0-recursion}, we get
	\begin{align}
    \label{eqn:maj-var-diff-recursion}
		\Var^+(S_{k+1}) - \Var^0(S_{k+1}) =&~ d \lambda \left(\Var^+(S_k) - \Var^0(S_k) \right)\\
		&~+ d \eta^2 (d\lambda)^{2k} (\lambda - \lambda^2 \mathbbm{1}_R). \nonumber\\
    \label{eqn:maj-var-sum-recursion}
		\Var^+(S_{k+1}) + \frac{q-2}2 \Var^0(S_{k+1})
		=&~ d \left(\Var^+(S_k) + \frac{q-2}2 \Var^0(S_{k+1})\right) \\
		&~+ d \eta^2 (d\lambda)^{2k} (1-\lambda^2 \mathbbm{1}_R). \nonumber
	\end{align}
	Solving \eqref{eqn:maj-var-diff-recursion}\eqref{eqn:maj-var-sum-recursion} we get
	\begin{align}
    \label{eqn:maj-var-diff-solution}
		&~\Var^+(S_{k}) - \Var^0(S_{k})\\
		=&~ \left(\Var^+(S_0) - \Var^0(S_0)\right) (d\lambda)^k
		+ \sum_{1\le i\le k} d \eta^2(d\lambda)^{2i-2}(d\lambda)^{k-i} (\lambda - \lambda^2 \mathbbm{1}_R) \nonumber\\
		=&~ O\left((d\lambda)^k\right) + d \eta^2 (d\lambda)^{k-1} \frac{(d\lambda)^k-1}{d\lambda-1} (\lambda - \lambda^2 \mathbbm{1}_R) \nonumber\\
		=&~ (1+o(1)) \frac{1-\lambda \mathbbm{1}_R}{d\lambda-1} \eta^2 (d\lambda)^{2k}\nonumber
  \end{align}
  and
  \begin{align}
    \label{eqn:maj-var-sum-solution}
		&~\Var^+(S_k) + \frac{q-2}2 \Var^0(S_k) \\
		=&~ \left(\Var^+(S_0) + \frac{q-2}2 \Var^0(S_0)\right) d^k
		+ \sum_{1\le i\le k} d \eta^2 (d\lambda)^{2i-2} d^{k-i} (1-\lambda^2 \mathbbm{1}_R) \nonumber\\
		=&~ O\left(d^k\right) + \eta^2 d^k \frac{(d\lambda^2)^k-1}{d\lambda^2-1} (1-\lambda^2 \mathbbm{1}_R) \nonumber\\
		=&~ (1+o(1)) \frac{1-\lambda^2 \mathbbm{1}_R}{d\lambda^2-1} \eta^2 (d\lambda)^{2k}. \nonumber
	\end{align}
  Combining \eqref{eqn:maj-var-diff-solution}\eqref{eqn:maj-var-sum-solution} we have
	\begin{align}
		\Var^+(S_k) = (1+o(1)) \frac 2q \left(\frac{1-\lambda^2 \mathbbm{1}_R}{d\lambda^2-1} + \frac{1-\lambda \mathbbm{1}_R}{d\lambda-1} \cdot \frac{q-2}2\right) \eta^2 (d\lambda)^{2k}.
	\end{align}

	Now we compute moments of $S_k \circ P_\lambda$.
	\begin{align}
		\bE^+[S_k \circ P_\lambda] = \lambda \eta (d\lambda)^k.
	\end{align}
	\begin{align}
		\bE^+[S_k^2 \circ P_\lambda] &= \left(\lambda + \frac{1-\lambda}q \cdot 2\right) \bE^+[S_k^2] + \frac{1-\lambda}q \cdot (q-2) \bE^0[S_k^2]\\
		&= \left(\lambda + \frac{1-\lambda}q \cdot 2\right) \left(\Var^+(S_k) + (\bE^+ S_k)^2\right) + \frac{1-\lambda}q \cdot (q-2) \Var^0(S_k) \nonumber\\
		&= \lambda (1+o(1)) \frac 2q \left(\frac{1-\lambda^2\mathbbm{1}_R}{d\lambda^2-1} + \frac{1-\lambda\mathbbm{1}_R}{d\lambda-1} \cdot \frac{q-2}2\right) \eta^2 (d\lambda)^{2k} \nonumber\\
		&+ \left(\lambda + \frac{1-\lambda}q \cdot 2\right) \eta^2 (d\lambda)^{2k}
		+ \frac{1-\lambda}q \cdot 2 \cdot (1+o(1)) \frac{1-\lambda^2 \mathbbm{1}_R}{d\lambda^2-1} \eta^2 (d\lambda)^{2k} \nonumber \\
		% & \le (1+o(1)) \left(\lambda + \frac{1-\lambda}q \cdot 2\right) \frac{d\lambda^2-\lambda^2\mathbbm{1}_R}{d\lambda^2-1} \eta^2 (d\lambda)^{2k} \nonumber.
    & = (1+o(1))\left(\frac 2q \cdot \frac{d\lambda^2-\lambda^2\mathbbm{1}_R}{d\lambda^2-1} + \lambda \cdot \frac{q-2}q \cdot \frac{d\lambda-\lambda \mathbbm{1}_R}{d\lambda-1}\right) \eta^2 (d\lambda)^{2k} \nonumber \\
    % & = (1+o(1))\left(\frac 2q \cdot \frac{1}{d\lambda^2-1} + \lambda \cdot \frac{q-2}q \cdot \frac 1{d\lambda^2-\lambda}\right) (d-\mathbbm{1}_R) \lambda^2 \eta^2 (d\lambda)^{2k}. \nonumber
    & = (1+o(1))c(q,\lambda,d)^{-1} \frac{d-\mathbbm{1}_R}{d\lambda^2-1} \lambda^2 \eta^2 (d\lambda)^{2k}. \nonumber
	\end{align}
	% (The inequality step uses $\frac{1-\lambda\mathbbm{1}_R}{d\lambda-1} \le \frac{1-\lambda^2\mathbbm{1}_R}{d\lambda^2-1}$.)
	Finally,
	\begin{align}
		C_{\chi^2}((M_k^U \circ P_\lambda)^R)
    \ge \frac{\left(\bE^+[S_k \circ P_\lambda]\right)^2}{\bE^+[S_k^2 \circ P_\lambda]}
		=  (1+o(1)) c(q,\lambda,d) \cdot \frac{d\lambda^2-1}{d-\mathbbm{1}_R}.
		% & \ge \frac{(1+o(1)) (d\lambda^2-1)}{d-\mathbbm{1}_R}. \nonumber
	\end{align}
\end{proof}

\subsection{Bounds on key constants} \label{sec:proof-bi:const}
In this section we prove bounds on key constants used in Theorem~\ref{thm:bi-low-snr} and~\ref{thm:bi-high-snr}.

\begin{proposition} \label{prop:bi-low-snr-C-bound}
  For $q\in \bZ_{\ge 2}$, $\lambda \in \left[-\frac 1{q-1},1\right]$, we have $C^L(q,\lambda) \le q^2$,
  where $C^L(q,\lambda)$ is defined in~\eqref{eqn:thm:bi-low-snr-defn-C}.
\end{proposition}
\begin{proof}
We have
\begin{align}
  &~\frac{\left\langle \left(\lambda\pi+\frac{1-\lambda}q\right)^{-1}+\frac 1q \left(\lambda\pi+\frac{1-\lambda}q\right)^{-2}, v^2\right\rangle}{\left\langle \pi^{-1} + \frac 1q\pi^{-2}, v^2 \right\rangle} \\
  \le&~ \max\left\{\frac{\left\langle \left(\lambda\pi+\frac{1-\lambda}q\right)^{-1},v^2\right\rangle}{\left\langle \pi^{-1}, v^2 \right\rangle}, \frac{\left\langle\left(\lambda\pi+\frac{1-\lambda}q\right)^{-2}, v^2\right\rangle}{\left\langle\pi^{-2},v^2\right\rangle}\right\} \nonumber \\
  \le&~ \max\{q,q^2\}=q^2. \nonumber
\end{align}
where the second step is by Lemma~\ref{lem:quad-form-ratio}.
\end{proof}

\begin{proposition} \label{prop:bi-high-snr-C-bound}
  For $q\in \bZ_{\ge 2}$, $\lambda \in \left[-\frac 1{q-1},1\right]$, we have
  $C^H(q,\lambda) \le q^{5/2}$,
  where $C^H(q,\lambda)$ is defined in~\eqref{eqn:thm:bi-high-snr-defn-C}.
\end{proposition}\begin{proof}
	By \eqref{eqn:h2-hess-ratio-step-1},
	\begin{align} \label{eqn:prop-bi-high-snr-C-bound-denom}
		f(\pi, v) \ge \frac 1{\sqrt q} \left(\sum_{i\in [q]} \pi^{1/2}_i\right) \left\langle \pi^{-3/2}, v^2\right\rangle \ge \frac 1{\sqrt q} \left\langle \pi^{-3/2}, v^2\right\rangle.
	\end{align}
	On the other hand,
	\begin{align} \label{eqn:prop-bi-high-snr-C-bound-num}
		f\left(\lambda \pi + \frac {1-\lambda}{q}, v\right)
		& \le \left(\sum_{i\in [q]} \left(\lambda \pi_i + \frac {1-\lambda}{q}\right)^{1/2}\right) \left\langle \left(\lambda \pi + \frac {1-\lambda}{q}\right)^{-3/2}, v^2\right\rangle \\
    & \le \sqrt q \cdot \left\langle \left(\lambda \pi + \frac {1-\lambda}{q}\right)^{-3/2}, v^2\right\rangle. \nonumber
		% & \le \sqrt q  \cdot \left\langle \left(\frac{\pi}q\right)^{-3/2}, v^2\right\rangle \nonumber\\
		% & = q^2 \left\langle \pi^{-3/2}, v^2\right\rangle. \nonumber
	\end{align}
	Combining \eqref{eqn:prop-bi-high-snr-C-bound-denom}\eqref{eqn:prop-bi-high-snr-C-bound-num} we get
	\begin{align}
		% c = \sup_{\substack{\pi \in \cP([q]) \\ v \in \mathbbm{1}^\perp \subseteq \bR^q}}
    \frac{f\left(\lambda \pi + \frac {1-\lambda} q, v\right)}{f(\pi, v)} \le q \cdot \frac{\left\langle \left(\lambda \pi + \frac {1-\lambda}{q}\right)^{-3/2}, v^2\right\rangle}{\left\langle \pi^{-3/2}, v^2\right\rangle}\le q^{5/2}
	\end{align}
  where the last step is by Lemma~\ref{lem:quad-form-ratio}.
\end{proof}

The following lemma is the crucial step in the proof of Prop.~\ref{prop:bi-low-snr-C-bound} and~\ref{prop:bi-high-snr-C-bound}.
\begin{lemma} \label{lem:quad-form-ratio}
  For $q\in \bZ_{\ge 2}$, $\alpha \in \bR_{\ge 1}$, $\lambda \in \left[-\frac 1{q-1},1\right]$, we have
  \begin{align}
    \sup_{\substack{\pi \in \cP([q]) \\ v \in \mathbbm{1}^\perp \subseteq \bR^q}} \frac{\left\langle\left(\lambda \pi + \frac {1-\lambda} q\right)^{-\alpha}, v^2\right\rangle}{\left\langle \pi^{-\alpha}, v^2\right\rangle} \le q^\alpha.
  \end{align}
\end{lemma}

\begin{proof}
  We prove the ferromagnetic case $(\lambda \in [0,1])$ and antiferromagnetic case $(\lambda \in \left[-\frac 1{q-1},0\right])$ separately.

  \textbf{Ferromagnetic case $(\lambda\in [0,1])$.}
  In this case, we have
  \begin{align}
    \lambda x + \frac {1-\lambda} q \ge \frac xq
  \end{align}
  for all $x\in [0, 1]$.
  Therefore
  \begin{align}
    \left\langle\left(\lambda \pi + \frac {1-\lambda} q\right)^{-\alpha}, v^2\right\rangle \le \left\langle\left(\frac \pi q\right)^{-\alpha}, v^2\right\rangle = q^\alpha \left\langle \pi^{-\alpha}, v^2\right\rangle.
  \end{align}
  Note that we did not use the assumption that $v\in \mathbbm{1}^\perp$.

  \textbf{Antiferromagnetic case $\left(\lambda\in \left[-\frac 1{q-1},0\right]\right)$.}
  We would like to prove that
  \begin{align} \label{eqn:lemma-quad-form-ratio-diff}
    q^\alpha \left\langle \pi^{-\alpha}, v^2\right\rangle - \left\langle\left(\lambda \pi + \frac {1-\lambda} q\right)^{-\alpha}, v^2\right\rangle =: \left\langle b, v^2\right\rangle
  \end{align}
  is non-negative for all $\pi \in \cP([q])$, $v\in \mathbbm{1}^\perp$,
  where
  \begin{align}
    b := \left(\frac \pi q\right)^{-\alpha} - \left(\lambda \pi + \frac {1-\lambda} q\right)^{-\alpha}.
  \end{align}

  \textbf{Step 1.}
  We fix $\pi \in \cP([q])$ and determine the optimal $v\in \mathbbm{1}^\perp$ to plug in~\eqref{eqn:lemma-quad-form-ratio-diff}, reducing the statement to one involving $\pi$ only.

  If $\lambda x + \frac {1-\lambda}q \le \frac xq$ for some $x\in [0,1]$, then
  $x\ge \frac{1-\lambda}{1-\lambda q} \ge \frac{q}{2q-1} > \frac 12$.
  So there exists at most one $i$ such that
  $\lambda \pi_i + \frac {1-\lambda}q \le \frac {\pi_i} q$ (equivalently, $b_i \le 0$).
  We can wlog assume that $\pi_1 \ge \pi_2 \ge \cdots \ge \pi_q$.
  Then we know $\lambda \pi_i + \frac {1-\lambda}q > \frac {\pi_i} q$ (equivalently, $b_i > 0$) for all $2\le i\le q$.
  If $b_1\ge 0$, then $\left\langle b, v^2\right\rangle$ is non-negative for all $v$ and we are done.
  Therefore, it remains to consider the case $b_1 < 0$.

  If $v_1=0$, then $\left\langle b, v^2\right\rangle$ is non-negative.
  Therefore we can assume $v_1\ne 0$. By rescaling, we can assume that $v_1=1$.
  So $v_2+\cdots+v_q = -1$.
  Because $b_2,\ldots, b_q$ are all positive, to minimize $\sum_{2\le i\le q} b_i v_i^2$ under linear constraint $v_2+\cdots+v_q=-1$, the optimal choice is
  $v_i = -b_i^{-1} Z^{-1}$ for $2\le i\le q$ where $Z := \sum_{2\le i\le q} b_i^{-1}$.
  For this choice of $v$, we have
  \begin{align}
    \left\langle b, v^2\right\rangle = b_1 + \sum_{2\le i\le q} b_i \cdot (-b_i^{-1} Z^{-1})^2
    = b_1 + Z^{-1}.
  \end{align}
  Therefore, it remains to prove
  \begin{align}
    Z \le (-b_1)^{-1} \label{eqn:lemma-quad-form-ratio-after-step-1}
  \end{align}
  where $\pi_1\ge \cdots \ge \pi_q$, $b_1<0$, and $b_2,\ldots,b_q>0$.

  \textbf{Step 2.}
  We reduce to the case where $\pi_3=\cdots=\pi_q=0$.
  Note that
  \begin{align}
    Z = \sum_{2\le i\le q} b_i^{-1} =
    \sum_{2\le i\le q} \left(\left(\frac {\pi_i} q\right)^{-\alpha} - \left(\lambda \pi_i + \frac {1-\lambda} q\right)^{-\alpha}\right)^{-1}.
  \end{align}
  By Lemma~\ref{lem:quad-form-ratio-helper1}, for fixed $\pi_1$, the optimal choice (for maximizing $Z$) of
  $\pi_2,\ldots, \pi_q$ is $\pi_3 = \cdots = \pi_q=0$.

  Write $\pi_1 = 1-x$, $\pi_2 = x$ where $x\in \left[0, \frac {\lambda-\lambda q}{1-\lambda q}\right]$.
  Then
  \begin{align}
    b_1 &= \left(\frac {1-x} q\right)^{-\alpha} - \left(\lambda (1-x) + \frac {1-\lambda} q\right)^{-\alpha},\\
    Z &= \left(\left(\frac x q\right)^{-\alpha} - \left(\lambda x + \frac {1-\lambda} q\right)^{-\alpha}\right)^{-1}.
  \end{align}
  By rearranging terms in~\eqref{eqn:lemma-quad-form-ratio-after-step-1}, we reduce to proving
  \begin{align}
    \left(\frac x q\right)^{-\alpha} + \left(\frac {1-x} q\right)^{-\alpha} \ge \left(\lambda x + \frac {1-\lambda} q\right)^{-\alpha} + \left(\lambda (1-x) + \frac {1-\lambda} q\right)^{-\alpha}
    \label{eqn:lemma-quad-form-ratio-after-step-2}
  \end{align}
  for $q\in \bZ_{\ge 2}$, $\alpha \in \bR_{\ge 1}$, $\lambda \in \left[-\frac 1{q-1},0\right]$, $x\in \left[0, \frac{\lambda-\lambda q}{1-\lambda q}\right]$.

  \textbf{Step 3.}
  Let $g_{q,\alpha,x}(\lambda) := \left(\lambda x + \frac {1-\lambda} q\right)^{-\alpha} + \left(\lambda (1-x) + \frac {1-\lambda} q\right)^{-\alpha}$ be the RHS of~\eqref{eqn:lemma-quad-form-ratio-after-step-2}.
  Then
  \begin{align*}
    g_{q,\alpha,x}''(\lambda) &= \alpha(\alpha+1) \left(x-\frac 1q\right)^2 \left(\lambda x + \frac {1-\lambda} q\right)^{-\alpha-2}\\
    &+ \alpha(\alpha+1) \left(1-x-\frac 1q\right)^2 \left(\lambda (1-x) + \frac {1-\lambda} q\right)^{-\alpha-2} \\
    &>0.
  \end{align*}
  So $g_{q,\alpha,x}$ is convex in $\lambda$.
  Therefore it suffices to verify~\eqref{eqn:lemma-quad-form-ratio-after-step-2} for $\lambda=0$ and $\lambda=-\frac 1{q-1}$.
  When $\lambda=0$, we have
  \begin{align}
    g_{q,\alpha,x}(\lambda) = \left(\frac 1q\right)^{-\alpha} + \left(\frac 1q\right)^{-\alpha} \le \left(\frac x q\right)^{-\alpha} + \left(\frac {1-x} q\right)^{-\alpha}.
  \end{align}
  When $\lambda = -\frac 1{q-1}$, we have
  \begin{align}
    g_{q,\alpha,x}(\lambda) = \left(\frac {1-x}{q-1}\right)^{-\alpha} + \left(\frac x{q-1}\right)^{-\alpha} \le \left(\frac x q\right)^{-\alpha} + \left(\frac {1-x} q\right)^{-\alpha}.
  \end{align}
  This finishes the proof.
\end{proof}

\begin{lemma} \label{lem:quad-form-ratio-helper1}
  For $q\in \bZ_{\ge 2}$, $\alpha \in \bR_{\ge 1}$, $\lambda \in \left[-\frac 1{q-1},0\right]$, the function
  \begin{align}
    f(x) := \left(\left(\frac x q\right)^{-\alpha} - \left(\lambda x + \frac {1-\lambda} q\right)^{-\alpha}\right)^{-1}.
  \end{align}
  is convex in $x\in [0, \frac{1-\lambda}{1-\lambda q}]$.
\end{lemma}
\begin{proof}
  Let $g(x) = \frac 1{f(x)}$.
  Then
  \begin{align}
    f''(x) = \frac{2g'(x)^2-g(x)g''(x)}{g(x)^3}.
  \end{align}
  It suffices to prove that
  \begin{align}
    2g'(x)^2-g(x)g''(x)\ge 0.
  \end{align}
  We have
  \begin{align}
    2g'(x)^2-g(x)g''(x)
    =&~2\alpha^2 \left(q^{-1}\left(\frac x q\right)^{-\alpha-1}-\lambda  \left(\lambda x + \frac {1-\lambda} q\right)^{-\alpha-1}\right)^2\\
    \nonumber  &-\left(\left(\frac x q\right)^{-\alpha} - \left(\lambda x + \frac {1-\lambda} q\right)^{-\alpha}\right) \\
    \nonumber &\cdot
    \alpha(\alpha+1)\left(q^{-2}\left(\frac x q\right)^{-\alpha-2}-\lambda^2 \left(\lambda x + \frac {1-\lambda} q\right)^{-\alpha-2}\right).
  \end{align}
  Write $u = \frac xq$, $v = \lambda x + \frac {1-\lambda}q$, $c = q\lambda$.
  Then we have $0\le u\le v\le 1$ and $-\frac{q}{q-1}\le c\le 0$.
  It suffices to prove
  \begin{align}
    2\alpha^2 (u^{-\alpha-1} - c v^{-\alpha-1})^2 - \alpha (\alpha+1)(u^{-\alpha} - v^{-\alpha})(u^{-\alpha-2}-c^2 v^{-\alpha-2}) \ge 0.
  \end{align}
  We have
  \begin{align}
    &~2\alpha^2 (u^{-\alpha-1} - c v^{-\alpha-1})^2 - \alpha (\alpha+1)(u^{-\alpha} - v^{-\alpha})(u^{-\alpha-2}-c^2 v^{-\alpha-2})\\
    \nonumber \ge&~\alpha (\alpha+1) ((u^{-\alpha-1} - c v^{-\alpha-1})^2 - (u^{-\alpha} - v^{-\alpha})(u^{-\alpha-2}-c^2 v^{-\alpha-2})) \\
    \nonumber =&~\alpha(\alpha+1)u^{-\alpha} v^{-\alpha}(u^{-1}-cv^{-1})^2\\
    \nonumber \ge &~0.
  \end{align}
  This finishes the proof.
\end{proof}

\section{Proofs in Section~\ref{sec:non-bi}} \label{sec:proof-non-bi}
\subsection{Local subadditivity} \label{sec:proof-non-bi:subadd}
In this section we prove Lemma~\ref{lem:subadd}.
We first prove the special case of FSCs.
\begin{lemma}\label{lem:subadd-fsc}
  Fix $q\in \bZ_{\ge 2}$. For any $\epsilon>0$ and $\pi,\pi'\in \cP([q])/\Aut([q])$ with $C_{\chi^2}(\FSC_{\pi'}) \le \epsilon$, we have
  \begin{align}
    C_{\chi^2} (\FSC_\pi \star \FSC_{\pi'}) \le (1+O_q(\epsilon^{1/2})) (C_{\chi^2}(\FSC_\pi) + C_{\chi^2}(\FSC_{\pi'})).
  \end{align}
\end{lemma}

\begin{proof}[Proof of Lemma~\ref{lem:subadd} given Lemma~\ref{lem:subadd-fsc}]
  Let $\pi_P$ (resp.~$\pi_Q$) be the $\pi$-component of $P$ (resp.~$Q$).

  Because the constant does not depend on $Q$, it suffices to prove the case where $Q$ is an FSC, i.e., $\pi_Q$ is fixed.

  If $C_{\chi^2}(Q) \le \epsilon^{2/5}$, then by Lemma~\ref{lem:subadd-fsc}, we have
  \begin{align}
    \nonumber C_{\chi^2}(P\star Q) &= \bE_{\pi_P} C_{\chi^2}(\FSC_{\pi_P} \star \FSC_{\pi_Q})\\
    \nonumber &  \le \bE_{\pi_P}\left[(1+O_q(\epsilon^{1/5})) (C_{\chi^2}(\FSC_{\pi_P}) + C_{\chi^2}(\FSC_{\pi_Q})) \right]\\
    & = (1+O_q(\epsilon^{1/5})) (C_{\chi^2}(P) + C_{\chi^2}(Q)).
  \end{align}

  In the following we assume that $C_{\chi^2}(Q) > \epsilon^{2/5}$.
  By Markov's inequality, we have
  \begin{align}
    \bP\left[ C_{\chi^2}(\FSC_{\pi_P}) \ge \epsilon^{2/5} \right] \le \epsilon^{3/5}.
    \label{eqn:thm-subadd-markov}
  \end{align}
  Write
  \begin{align}
    \nonumber C_{\chi^2}(P\star Q) =&~ \bE_{\pi_P} \left[C_{\chi^2}(\FSC_{\pi_P} \star \FSC_{\pi_Q}) \mathbbm{1}\{C_{\chi^2}(\FSC_{\pi_P}) \le \epsilon^{2/5}\}\right] \\
    \nonumber  &~+ \bE_{\pi_P} \left[C_{\chi^2}(\FSC_{\pi_P} \star \FSC_{\pi_Q}) \mathbbm{1}\{C_{\chi^2}(\FSC_{\pi_P}) > \epsilon^{2/5}\}\right] \\
    =:&~ L+R.
  \end{align}
  For $L$, by Lemma~\ref{lem:subadd-fsc}, we have
  \begin{align}
    \nonumber L &\le (1+O_q(\epsilon^{1/5})) \bE_{\pi_P} \left[(C_{\chi^2}(\FSC_{\pi_P}) + C_{\chi^2}(\FSC_{\pi_Q})) \mathbbm{1}\{C_{\chi^2}(\FSC_{\pi_P}) \le \epsilon^{2/5}\}\right] \\
    &\le (1+O_q(\epsilon^{1/5})) (C_{\chi^2}(P) + C_{\chi^2}(Q)).
    \label{eqn:thm-subadd-L-bound}
  \end{align}

  For $R$, by \eqref{eqn:thm-subadd-markov} and the assumption that $C_{\chi^2}(Q) > \epsilon^{2/5}$, we have
  \begin{align}
    \bE_{\pi_P} \left[C_{\chi^2}(\FSC_{\pi_P} \star \FSC_{\pi_Q}) \mathbbm{1}\{C_{\chi^2}(\FSC_{\pi_P}) > \epsilon^{2/5}\}\right] \le O_q(\epsilon^{3/5}) \le O_q(\epsilon^{1/5}) C_{\chi^2}(Q).
    \label{eqn:thm-subadd-R-bound}
  \end{align}

  Combining \eqref{eqn:thm-subadd-L-bound} and \eqref{eqn:thm-subadd-R-bound} we finish the proof.
\end{proof}

\begin{proof}[Proof of Lemma~\ref{lem:subadd-fsc}]
  % Note that statement of Lemma~\ref{lem:subadd-fsc} is monotone in $\epsilon$. So we can wlog assume that $C_{\chi^2}(\FSC_{\pi'}) = \epsilon$.
  Because the statement is monotone in $\epsilon$, we can wlog assume that $C_{\chi^2}(\FSC_{\pi'}) = \epsilon$.

  % Throughout this proof, statements like ``$f\le Cg$'' should be understood as ``there exists a constant $C_q$ depending only on $q$ such that $f\le C_q g$''.

  Let $\pi'_i = \frac{1+\epsilon_i} q$. Then
  \begin{align}
    \sum_i \epsilon_i = 0, \quad C_{\chi^2}(\FSC_{\pi'}) = \frac 1q \sum_i \epsilon_i^2 = \epsilon.
  \end{align}
  By \eqref{eqn:star-formula},
  \begin{align}
    \nonumber C_{\chi^2}(\FSC_{\pi} \star \FSC_{\pi'}) &= q \sum_{\tau \in S_q} \frac 1{(q-1)!} \cdot \frac{\sum_i \pi_i^2 \pi_{\tau(i)}^{\prime 2}}{\sum_i \pi_i \pi'_{\tau(i)}} -1\\
    &= q \sum_{\tau \in S_q} \frac 1{q!} \cdot \frac{\sum_i \pi_i^2 (1+\epsilon_{\tau(i)})^2}{1+\sum_i \pi_i \epsilon_{\tau(i)}} -1.
    \label{eqn:lem-subadd-fsc-1}
  \end{align}

  Recall the following basic equality.
  \begin{align}
    \frac 1{1+x} = 1-x+x^2 - \frac{x^3}{1+x}. \label{eqn:lem-subadd-fsc-basic}
  \end{align}
  We apply \eqref{eqn:lem-subadd-fsc-basic} with $x=\sum_i \pi_i \epsilon_{\tau(i)}$.
  Because $|x| = O_q(\epsilon^{1/2})$, we have
  \begin{align}
    \left|\frac{x^3}{1+x}\right| = O_q(\epsilon^{3/2}).
  \end{align}

  So
  \begin{align} %\label{eqn:lem-subadd-fsc-2}
    & \frac{\sum_i \pi_i^2 (1+\epsilon_{\tau(i)})^2}{1+\sum_i \pi_i \epsilon_{\tau(i)}} \nonumber \\
    &= \left(\sum_i \pi_i^2 (1+\epsilon_{\tau(i)})^2\right)\left(1-x+x^2 - \frac{x^3}{1+x}\right) \nonumber \\
    &\le \left(\sum_i \pi_i^2 (1+\epsilon_{\tau(i)})^2\right)\left(1-\sum_i \pi_i \epsilon_{\tau(i)}+\left(\sum_i \pi_i \epsilon_{\tau(i)}\right)^2 + O_q(\epsilon^{3/2})\right) \nonumber \\
    &\le \left(\sum_i \pi_i^2 (1+\epsilon_{\tau(i)})^2\right)\left(1-\sum_i \pi_i \epsilon_{\tau(i)}+\left(\sum_i \pi_i \epsilon_{\tau(i)}\right)^2\right) + O_q(\epsilon^{3/2}), \label{eqn:lem-subadd-fsc-3}
  \end{align}
  where the last step is by
  \begin{align}
    \sum_i \pi_i^2 (1+\epsilon_{\tau(i)})^2 = O(1).
  \end{align}

  Let us expand the first summand in \eqref{eqn:lem-subadd-fsc-3}.
  \begin{align}
    \nonumber &\left(\sum_i \pi_i^2 (1+\epsilon_{\tau(i)})^2\right)\left(1-\sum_i \pi_i \epsilon_{\tau(i)}+\left(\sum_i \pi_i \epsilon_{\tau(i)}\right)^2\right)\\
    &= \left(\sum_i \pi_i^2 + 2 \sum_i \pi_i^2 \epsilon_{\tau(i)} + \sum_i \pi_i^2 \epsilon_{\tau(i)}^2\right)\left(1-\sum_i \pi_i \epsilon_{\tau(i)}+\left(\sum_i \pi_i \epsilon_{\tau(i)}\right)^2\right)\nonumber \\
    &=: (\circnum{1} + \circnum{2} + \circnum{3})(1 - \circnum{4} + \circnum{5}) \nonumber\\
    &= \circnum{1}-\circnum{1}\circnum{4}+\circnum{1}\circnum{5}
    +\circnum{2}-\circnum{2}\circnum{4}+\circnum{2}\circnum{5}
    +\circnum{3}-\circnum{3}\circnum{4}+\circnum{3}\circnum{5}.
    \label{eqn:lem-subadd-fsc-4}
    % &= \left(\sum_i \pi_i^2\right)
    % + \left(2 \sum_i \pi_i^2 \epsilon_{\tau(i)}\right)
    % -\left(\sum_i \pi_i^2\right)\left(\sum_i \pi_i \epsilon_{\tau(i)}\right) \nonumber \\
    % &+\left(\sum_i \pi_i^2 \epsilon_{\tau(i)}^2\right)
    % -\left(2 \sum_i \pi_i^2 \epsilon_{\tau(i)}\right)\left(\sum_i \pi_i \epsilon_{\tau(i)}\right)
    % +\left(\sum_i \pi_i^2\right)\left(\sum_i \pi_i \epsilon_{\tau(i)}\right)^2 \nonumber \\
    % &+ C \epsilon^{3/2}. \nonumber
  \end{align}

  Note that we have the following loose bounds:
  \begin{align}
    \circnum{1} = O_q(1), \quad |\circnum{2}| = O_q(\epsilon^{1/2}), \quad \circnum{3} \le O_q(\epsilon), \quad |\circnum{4}| \le O_q(\epsilon^{1/2}), \quad \circnum{5} \le O_q(\epsilon). \label{eqn:lem-subadd-fsc-loose}
  \end{align}

  Let us study every term under $\sum_{\tau\in S_q} \frac 1{q!}$.
  For simplicity, write
  \begin{align}
    A = \sum_i \pi_i^2 = \frac 1q\left(1+C_{\chi^2}(\FSC_{\pi})\right).
  \end{align}

  \circnum{1}:
  \begin{align} \label{eqn:lem-subadd-fsc-c1}
    \sum_{\tau\in S_q} \frac 1{q!} \cdot \circnum{1} = \sum_{\tau\in S_q} \frac 1{q!} \sum_i \pi_i^2 = A.
  \end{align}

  \circnum{1}\circnum{4}:
  \begin{align} \label{eqn:lem-subadd-fsc-c1c4}
    \sum_{\tau\in S_q} \frac 1{q!} \cdot \circnum{1}\circnum{4} = \sum_{\tau\in S_q} \frac 1{q!} \left(\sum_i \pi_i^2\right) \left(\sum_i \pi_i \epsilon_{\tau(i)}\right)= 0.
  \end{align}

  \circnum{1}\circnum{5}:
  \begin{align}
    \nonumber \sum_{\tau\in S_q} \frac 1{q!} \cdot \circnum{1}\circnum{5}
    &=\sum_{\tau\in S_q} \frac 1{q!} \left(\sum_i \pi_i^2\right)\left(\sum_i \pi_i \epsilon_{\tau(i)}\right)^2 \\
    &=A \sum_{i,j} \sum_{\tau\in S_q} \frac 1{q!} \cdot \pi_i \pi_j \epsilon_{\tau(i)} \epsilon_{\tau(j)} \nonumber \\
    &=A \sum_{i,j} \epsilon_i \epsilon_j \sum_{\tau\in S_q} \frac 1{q!} \cdot \pi_{\tau(i)} \pi_{\tau(j)} \nonumber \\
    &=A \sum_{i,j} \epsilon_i \epsilon_j \left\{
      \begin{array}[]{ll}
        \frac 1q \sum_k \pi_k^2 & i=j,\\
        \frac 1{q(q-1)} \left(1-\sum_k \pi_k^2\right) & i\ne j,
      \end{array}
    \right. \nonumber \\
    &= A \left(\sum_i \epsilon_i^2 \cdot \frac 1q \sum_k \pi_k^2 + \sum_i \epsilon_i (-\epsilon_i)\cdot \frac 1{q(q-1)} \left(1-\sum_k \pi_k^2\right) \right) \nonumber \\
    & = A \cdot \frac 1q \sum_i \epsilon_i^2 \left(\sum_k \pi_k^2 - \frac 1{q-1} \left(1-\sum_k \pi_k^2\right) \right) \nonumber \\
    &= \epsilon A \cdot \frac{q A-1}{q-1}.
    \label{eqn:lem-subadd-fsc-c1c5}
  \end{align}

  \circnum{2}:
  \begin{align} \label{eqn:lem-subadd-fsc-c2}
    \sum_{\tau\in S_q} \frac 1{q!} \cdot \circnum{2} = \sum_{\tau\in S_q} \frac 1{q!} \cdot 2 \sum_i \pi_i^2 \epsilon_{\tau(i)} = 0.
  \end{align}

  \circnum{2}\circnum{4}:
  \begin{align}
    \nonumber \sum_{\tau\in S_q} \frac 1{q!} \cdot \circnum{2}\circnum{4}
    &=\sum_{\tau\in S_q} \frac 1{q!} \left(2 \sum_i \pi_i^2 \epsilon_{\tau(i)}\right)\left(\sum_i \pi_i \epsilon_{\tau(i)}\right) \\
    &= \sum_{i,j} \sum_{\tau\in S_q} \frac 1{q!} \cdot 2 \pi_i^2 \pi_j \epsilon_{\tau(i)} \epsilon_{\tau(j)} \nonumber \\
    &=\sum_{i,j} 2 \epsilon_i \epsilon_j \sum_{\tau\in S_q} \frac 1{q!} \cdot \pi_{\tau(i)}^2 \pi_{\tau(j)} \nonumber \\
    &=\sum_{i,j} 2 \epsilon_i \epsilon_j \left\{
      \begin{array}{ll}
        \frac 1q \sum_k \pi_k^3 & i=j,\\
        \frac 1{q(q-1)} \sum_k \pi_k^2(1-\pi_k) & i\ne j,
      \end{array}\right. \nonumber \\
    &=2\sum_i \epsilon_i^2 \cdot \frac 1q \sum_k \pi_k^3 + 2\sum_i \epsilon_i(-\epsilon_i) \cdot \frac 1{q(q-1)} \sum_k \pi_k^2(1-\pi_k) \nonumber \\
    &=2 \cdot \frac 1q \sum_i \epsilon_i^2 \left(\sum_k \pi_k^3 - \frac 1{q-1} \sum_k \pi_k^2(1-\pi_k)\right) \nonumber \\
    &=2 \epsilon \cdot  \frac 1{q-1}\left(q\sum_k \pi_k^3 - \sum_k \pi_k^2\right) \ge 0,
    \label{eqn:lem-subadd-fsc-c2c4}
  \end{align}
  where the last step is by
  \begin{align}
    \sum_k \pi_k^3= \left(\sum_k \pi_k^3\right)\left(\sum_k \pi_k\right) \ge \left(\sum_k \pi_k^2\right)^2 \ge \frac 1q \left(\sum_k \pi_k^2\right).
  \end{align}

  \circnum{2}\circnum{5}:
  By \eqref{eqn:lem-subadd-fsc-loose},
  \begin{align} \label{eqn:lem-subadd-fsc-c2c5}
    \left|\sum_{\tau\in S_q} \frac 1{q!} \cdot \circnum{2}\circnum{5} \right| = O_q(\epsilon^{3/2}).
  \end{align}

  \circnum{3}:
  \begin{align} \label{eqn:lem-subadd-fsc-c3}
    \sum_{\tau\in S_q} \frac 1{q!} \cdot \circnum{3} = \sum_{\tau\in S_q} \frac 1{q!} \left(\sum_i \pi_i^2 \epsilon_{\tau(i)}^2\right) = \sum_i \pi_i^2 \cdot \frac 1q \sum_j \epsilon_j^2 = \epsilon A.
  \end{align}

  \circnum{3}\circnum{4}:
  By \eqref{eqn:lem-subadd-fsc-loose},
  \begin{align} \label{eqn:lem-subadd-fsc-c3c4}
    \left|\sum_{\tau\in S_q} \frac 1{q!} \cdot \circnum{3}\circnum{4}\right| = O_q(\epsilon^{3/2}).
  \end{align}

  \circnum{3}\circnum{5}:
  By \eqref{eqn:lem-subadd-fsc-loose},
  \begin{align} \label{eqn:lem-subadd-fsc-c3c5}
    \left|\sum_{\tau\in S_q} \frac 1{q!} \cdot \circnum{3}\circnum{5}\right| = O_q(\epsilon^2).
  \end{align}

  Plugging \eqref{eqn:lem-subadd-fsc-c1} - \eqref{eqn:lem-subadd-fsc-c3c5} into \eqref{eqn:lem-subadd-fsc-4}\eqref{eqn:lem-subadd-fsc-3}\eqref{eqn:lem-subadd-fsc-1}, we get
  \begin{align}
    \nonumber C_{\chi^2}(\FSC_{\pi} \star \FSC_{\pi'})
    &\le q\left(A+\epsilon A \cdot \frac{qA-1}{q-1} + \epsilon A + O_q(\epsilon^{3/2})\right)-1\\
    \nonumber &=(qA-1)\left(1+\frac{q\epsilon A}{q-1} + \epsilon\right) + \epsilon + O_q(\epsilon^{3/2})\\
    &=\left(1+\frac{q\epsilon A}{q-1} + \epsilon\right) C_{\chi^2}(\FSC_{\pi}) + (1+O_q(\epsilon^{1/2})) C_{\chi^2}(\FSC_{\pi'}) \nonumber \\
    &=(1+O_q(\epsilon^{1/2})) (C_{\chi^2}(\FSC_{\pi}) + C_{\chi^2}(\FSC_{\pi'})),
    \label{eqn:lem-subadd-fsc-final}
  \end{align}
  where the last step is by $A\le 1$.
  % Because $A \le 1$, for $\epsilon$ small enough (it suffices to take $\epsilon = c \delta^2$ for a small enough constant $c=c(q)>0$), we finish the proof.
\end{proof}

\subsection{Proof of Prop.~\ref{prop:no-robust-recon}} \label{sec:proof-non-bi:no-robust-recon}
\begin{proof}[Proof of Prop.~\ref{prop:no-robust-recon}]
  We treat the regular tree case and the Poisson tree case uniformly.
  Let $D$ be the offspring distribution, i.e., $D=\mathbbm{1}_d$ for the regular case and $D=\Pois(d)$ for the Poisson case.

  Let $C_1>0$ be the constant in Lemma~\ref{lem:subadd}, i.e., for all $\epsilon>0$, FMS channels $P: \cX\to \cY$, $Q:\cX\to \cZ$ with $C_{\chi^2}(P)\le \epsilon$, we have
  \begin{align} \label{eqn:proof-thm-no-robust-recon-local-subadd}
    C_{\chi^2}(P \star Q) \le \left(1+C_1 \epsilon^{1/5}\right) (C_{\chi^2}(P) + C_{\chi^2}(Q)).
  \end{align}
  Let $C_2>0$ be such that $C_{\chi^2}(P) \le C_2$ for all FMS channels $P$.

  Take $c_1,c_2>0$ such that
  \begin{align} \label{eqn:proof-thm-no-robust-recon-c1-c2}
    \exp(c_1) d\lambda^2 + c_2 < 1.
  \end{align}

  Take $b_0$ such that for all $t\ge b_0$, we have
  \begin{align} \label{eqn:proof-thm-no-robust-recon-tail-bound}
    t^5 \bP_{b\sim D}[b>t] < c_2 C_2^{-1} (c_1 C_1^{-1})^5.
  \end{align}
  For the regular case we can take $b_0=d$. For the Poisson case the existence of such $b_0$ follows from Poisson tail behavior.

  For $\epsilon>0$, define
  \begin{align}
    b(\epsilon) := c_1 C_1^{-1} \epsilon^{-1/5}.
  \end{align}
  Take $\epsilon_0>0$ such that $b(\epsilon_0) > b_0$.

  Take $\epsilon_1>0$ such that $\epsilon_1<\min\{\epsilon_0, \delta\}$.
  Take $\epsilon>0$ such that $\epsilon < \epsilon_1$ and
  \begin{align}
    (\exp(c_1) d\lambda^2 + c_2)\epsilon_1 + \exp(c_1)\epsilon < \epsilon_1.
  \end{align}

  Let $W$ (resp.~$P$) be a $q$-FMS channel satisfying $C_{\chi^2}(W) \le \epsilon$ (resp.~$C_{\chi^2}(P) \le \epsilon_1$).
  We have
  \begin{align}
    \nonumber C_{\chi^2}(\BP_W(P)) =&~ \bE_{b\sim D} C_{\chi^2}((P\circ P_\lambda)^{\star b}\star W) \\
    \nonumber =&~ \bE_{b\sim D} \left[C_{\chi^2}((P\circ P_\lambda)^{\star b} \star W) \mathbbm{1} \{b \le b(\epsilon_1)\}\right]\\
    \nonumber &~+ \bE_{b\sim D} \left[C_{\chi^2}((P\circ P_\lambda)^{\star b} \star W) \mathbbm{1} \{b > b(\epsilon_1)\}\right] \\
    =:&~ L+R.
    \label{eqn:proof-prop-no-robust-recon-sum-L-R}
  \end{align}

  For $L$, by induction on $b$ we have
  \begin{align}
    C_{\chi^2}((P\circ P_\lambda)^{\star b}\star W) \le (1+C_1 \epsilon_1^{1/5})^b (b C_{\chi^2}(P\circ P_\lambda)+\epsilon) \le (1+C_1 \epsilon_1^{1/5})^b (b \lambda^2 \epsilon_1+\epsilon).
  \end{align}
  Then
  \begin{align}
    \nonumber L&\le \bE_{b\sim D} \left[\exp(C_1 \epsilon_1^{1/5} b) (b \lambda^2 \epsilon_1+\epsilon) \mathbbm{1} \{b \le b(\epsilon_1)\}\right] \\
    \nonumber &\le \bE_{b\sim D} \left[\exp(C_1 \epsilon_1^{1/5} b(\epsilon_1)) (b \lambda^2 \epsilon_1+\epsilon) \right] \\
    \nonumber &\le \exp(C_1 \epsilon_1^{1/5} b(\epsilon_1)) (d \lambda^2 \epsilon_1+\epsilon) \\
    &=\exp(c_1) (d \lambda^2 \epsilon_1+\epsilon).
    \label{eqn:proof-prop-no-robust-recon-L-bound}
  \end{align}

  For $R$ we have
  \begin{align}
    R \le C_2 \cdot c_2 C_2^{-1} \left(c_1 C_1^{-1}\right)^5 \cdot b(\epsilon_1)^{-5}
    = c_2 \epsilon_1.
    \label{eqn:proof-prop-no-robust-recon-R-bound}
  \end{align}

  Combining \eqref{eqn:proof-prop-no-robust-recon-sum-L-R}\eqref{eqn:proof-prop-no-robust-recon-L-bound}\eqref{eqn:proof-prop-no-robust-recon-R-bound} we get
  \begin{align} \label{eqn:proof-prop-no-robust-recon-contraction}
    C_{\chi^2}(\BP_W(P)) \le (\exp(c_1) d\lambda^2 +c_2)\epsilon_1 + \exp(c_1)\epsilon < \epsilon_1.
  \end{align}

  Let $M_k$ be the channel $\sigma_\rho \mapsto (T_k, \omega_{T_k})$.
  Then $M_{k+1} = \BP_W(M_k)$.
  By \eqref{eqn:proof-prop-no-robust-recon-contraction} and $\epsilon\le \epsilon_1$ we see that
  \begin{align}
    C_{\chi^2}(M_k) \le \epsilon_1 < \delta
  \end{align}
  for all $k$. This finishes the proof.
\end{proof}

\section{Asymmetric fixed points} \label{sec:asymm}
Up to now we have focused on symmetric fixed points of the $\BP$ operator.
% In this section we discuss asymmetric fixed points of the BP operator.
If we view the $\BP$ operator as an operator on the space of $q$-ary input (possibly asymmetric) channels, then a natural question to determine the (possibly asymmetric) fixed points.
In the case $q=2$, \cite{yu2022ising} showed that there is only one non-trivial fixed point, and the fixed point is symmetric.
For $q\ge 3$, it is no longer the case.
\begin{proposition} \label{prop:bp-fixed-point-asymm}
  Work under the setting of Theorem~\ref{thm:bi-imprecise}.
  If $q\ge 3$ and $d\lambda^2>1$, then the $\BP$ operator \eqref{eqn:bp-operator} has at least one non-trivial asymmetric fixed point.
\end{proposition}
\begin{proof}
  Consider the channel $U: [q]\to \{\pm\}$, which maps $1$ to $+$ and $2,\ldots,q$ to $-$.
  Because $\BP(U)\le_{\deg} U$, the sequence $(\BP^k(U))_{k\ge 0}$ is non-increasing in degradation preorder.
  Therefore a limit channel $\BP^\infty(U)$ exists by Lemma~\ref{lem:deg-down-seq-limit}.
  % This means that for any convex potential function $\phi$ on the probability simplex $\cP([q])$, the sequence $(\Phi(\BP^k(U)))_{k\ge 0}$ is non-increasing (where $\Phi$ is the induced function on the space of $q$-ary input channels).
  % By monotone convergence theorem, the limit
  % \begin{align}
  %   \lim_{k\to \infty} \Phi(\BP^k(U))
  % \end{align}
  % exists for all convex potential functions $\phi$.
  % These limit potentials are compatible with each other and uniquely determines the limit channel $\BP^\infty(U)$.

  We would like to show that $\BP^\infty(U)$ is a non-FMS non-trivial fixed point.
  When $d\lambda^2>1$, count-reconstruction is possible (see e.g.~\cite{mossel2001reconstruction}).
  So it is possible to gain non-trivial information about whether the input is $1$ by counting the number of $+$.
  So $\BP^\infty(U)$ is non-trivial.

  On the other hand, $\BP^\infty(U) (\cdot | i)$ are the same for $i=2,\ldots, q$.
  This cannot happen for any non-trivial FMS channel.
  Therefore $\BP^\infty(U)$ is not an FMS channel.
\end{proof}

Nevertheless, when the condition in Theorem~\ref{thm:bi-imprecise} holds, for an open set of initial channels, it will converge to the unique FMS fixed point under BP iterations.
We make the following definition.
\begin{definition}
  Let $U: \cX \to \cY$ be a channel where $\cX = [q]$.
  We say $U$ has full rank if there exists a partition of
  $\cY$ into $q$ measurable subsets $\cY = E_1 \cup \cdots \cup E_q$ such that the $q\times q$ matrix
  \begin{align}
    (U(E_j | i))_{i\in [q], j\in [q]}
  \end{align}
  is invertible.
\end{definition}
\begin{proposition} \label{prop:bp-fixed-point-full-rank}
  Work under the setting of Theorem~\ref{thm:bi-imprecise}.
  If $(q,\lambda,d)$ satisfies \eqref{eqn:thm:bi-low-snr-cond} or \eqref{eqn:thm:bi-high-snr-cond}, then for any $q$-ary input (possibly asymmetric) channel $U$ of full rank, we have
  \begin{align}
    \BP^\infty(U) = \BP^\infty(\Id).
  \end{align}
\end{proposition}
\begin{proof}
  We prove that under the condition that $U$ has full rank, there exists a channel $R$ such that $R\circ U$ is a non-trivial Potts channel.

  Because $U$ has full rank, there exists a partition $\cY = E_1 \cup \cdots \cup E_q$ such that
  \begin{align}
    (U(E_j | i))_{i\in [q], j\in [q]}
  \end{align}
  is invertible.
  Define $Q: \cY \to [q]$ by mapping $y\in E_i$ to $i$ for all $i\in [q]$.
  Then we can replace $U$ with $Q \circ U$ and wlog assume that $\cY=[q]$.

  By Lemma~\ref{lem:sing-deg-potts}, there exists $\lambda>0$ such that $P_\lambda \le_{\deg} U$.
  Therefore $P_\lambda \le_{\deg} U \le_{\deg} \Id$.
  Degradation of $q$-ary input (possibly asymmetric) channels is preserved under $\BP$ operator.
  So by iterating the $\BP$ operator, we get
  \begin{align}
    \BP^\infty(P_\lambda) \le_{\deg} \BP^\infty(U) \le_{\deg} \BP^\infty(\Id).
  \end{align}
  The first and third channels are equal by Theorem~\ref{thm:bi-imprecise}.
  Therefore $\BP^\infty(U) = \BP^\infty(\Id)$.
\end{proof}

\begin{lemma} \label{lem:sing-deg-potts}
  Fix $q\in \bZ_{\ge 2}$ and $\epsilon>0$.
  Then there exists $\lambda>0$ such that for any probability kernel $U:[q]\to [q]$ with $\sigma_{\min}(U) > \epsilon$, we have $P_\lambda \le_{\deg} U$.
\end{lemma}
\begin{proof}
  Because $\sigma_{\min}(U) > \epsilon$, we have
  \begin{align}
    \max_{i,j\in [q]} \left|\left(U^{-1}\right)_{i,j}\right| \le \|U^{-1}\|_2 \le \sqrt q \epsilon^{-1}.
  \end{align}
  Let $J\in \bR^{q\times q}$ be the all ones matrix.
  Because $U$ is a stochastic matrix, we have $U^{-1} J=J$.
  Because the maximum (in absolute value) entry of $U^{-1}$ is bounded by an constant, for some constant $\lambda>0$ the matrix $U^{-1} P_\lambda$ has non-negative entries.
  Note that $U^{-1} P_\lambda \mathbbm{1} = U^{-1} \mathbbm{1} = \mathbbm{1}$.
  So $U^{-1} P_\lambda$ is a stochastic matrix.
  Let $R=U^{-1} P_\lambda$.
  Then $R\circ U = UR = P_\lambda$, thus $P_\lambda\le_{\deg} U$.
\end{proof}

For the boundary irrelevance operator $\BP_W$ \eqref{eqn:bp-w-operator}, the situation is simpler: when the survey FMS channel $W$ is non-trivial, there is no asymmetric fixed point.
\begin{proposition} \label{prop:bp-prime-fixed-point-asymm}
  Work under the setting of Theorem~\ref{thm:bi-imprecise}.
  If $(q,\lambda,d,W)$ satisfies \eqref{eqn:thm:bi-low-snr-cond} or \eqref{eqn:thm:bi-w-high-snr-cond}, then $\BP_W$ has only one fixed point.
\end{proposition}
\begin{proof}
  Suppose $U$ is a fixed point of $\BP_W$. We have
  \begin{align}
    0 \le_{\deg} U \le_{\deg} \Id.
  \end{align}
  Degradation for $q$-ary input (possibly asymmetric) channels is preserved under $\BP_W$ operator.
  So by iterating the $\BP_W$ operator, we get
  \begin{align}
    \BP_W^{\infty} (0) \le_{\deg} \BP_W^{\infty}(U) \le_{\deg} \BP_W^{\infty}(\Id).
  \end{align}
  The first and third channels are equal by Theorem~\ref{thm:bi-imprecise}.
  Therefore $U = \BP_W^{\infty}(U) = \BP_W^{\infty}(\Id)$ is equal to the unique FMS fixed point.
\end{proof}

\section{SBM mutual information} \label{sec:mutual-info}
In this section we prove Theorem~\ref{thm:sbm-mutual-info}.
The proof is a direct generalization of \cite[Theorem 1]{abbe2021stochastic}.
\begin{proof}[Proof of Theorem~\ref{thm:sbm-mutual-info}]
  Let $Y_v^\epsilon \sim \EC_\epsilon(\cdot | X_v)$ for $v\in V$ and $\epsilon\in [0,1]$.
  Let $u\in V$ be a fixed vertex.
  Define $f(\epsilon) := \frac 1n I(X; G, Y^\epsilon)$.
  Then $f(0) = H(X_u)$ and $f(1) = \frac 1n I(X; G)$.
  Furthermore, calculation shows that
  \begin{align}
    f'(\epsilon) = -H(X_u | G, Y^\epsilon_{V\backslash u}).
  \end{align}
  Let $k\in \bZ_{\ge 1}$ be a constant, $B(u,k)$ be the set of vertices with distance $\le k$ to $u$, and
  $\partial B(u,k)$ be the set of vertices at distance $k$ to $u$.
  By the data processing inequality and Lemma \ref{lem:sbm-approx-cond-indep}, we have
  \begin{align}
    I(X_u; G, Y^\epsilon_{B(u,k) \backslash u}) \le I(X_u; G, Y^\epsilon_{V\backslash u}) \le I(X_u; G, Y^\epsilon_{B(u,k)\backslash u}, X_{\partial B(u,k)})+o(1).
  \end{align}
  By Lemma \ref{lem:sbm-bot-coupling}, we have
  \begin{align}
    I(\sigma_\rho; \omega^\epsilon_{T_k \backslash \rho} | T_k) - o(1) \le I(X_u; G, Y^\epsilon_{V\backslash u}) \le I(\sigma_\rho; \omega^\epsilon_{T_k \backslash \rho}, \sigma_{L_k} | T_k) + o(1).
  \end{align}
  Taking limit $n\to \infty$, then taking limit $k\to \infty$, we get
  \begin{align}
    \lim_{k\to \infty} I(\sigma_\rho; \omega^\epsilon_{T_k \backslash \rho} | T_k) \le \lim_{n\to \infty} I(X_u; G, Y^\epsilon_{V\backslash u}) \le \lim_{k\to \infty} I(\sigma_\rho; \omega^\epsilon_{T_k \backslash \rho}, \sigma_{L_k} | T_k).
  \end{align}
  The first and third terms are equal by the boundary irrelevance assumption.
  Therefore
  \begin{align}
    \lim_{n\to\infty} I(X_u; G, Y^\epsilon_{V\backslash u}) = \lim_{k\to \infty} I(\sigma_\rho; \omega^\epsilon_{T_k \backslash \rho} | T_k)
  \end{align}
  So
  \begin{align}
    \lim_{n\to \infty} \frac 1n I(X; G) &= H(X_u) - \int_0^1 \lim_{n\to\infty} H(X_u | G, Y^\epsilon_{V\backslash u}) d\epsilon \\
    &=\int_0^1 \lim_{n\to \infty} I(X_u; G, Y^\epsilon_{V\backslash u}) d\epsilon  \nonumber \\
    &=\int_0^1 \lim_{k\to \infty} I(\sigma_\rho; \omega^\epsilon_{T_k \backslash \rho} | T_k) d\epsilon. \nonumber
  \end{align}
\end{proof}

\begin{lemma}[SBM-BOT coupling] \label{lem:sbm-bot-coupling}
  Let $(X,G)\sim \mathrm{SBM}(n,q,\frac an,\frac bn)$.
  Let $v\in V$ and $k=o(\log n)$. Let $B(v,k)$ be the set of vertices with distance $\le k$ to $v$, and
  $\partial B(v,k)$ be the set of vertices at distance $k$ to $v$.
  Let $d = \frac {a+(q-1)b}q$ and $\lambda = \frac{a-b}{a+(q-1)b}$.
  Let $(T,\sigma)$ be the Potts model with broadcasting channel $P_\lambda$ on a Poisson tree with expected offspring $d$. Let $\rho$ be the root of $T$, $L_k$ be the set of vertices at distance $k$ to $\rho$, $T_k$ be the set of vertices at distance $\le k$ to $\rho$.

  Then $(G|_{B(v,k)}, X_{B(v,k)})$ can be coupled (with $o(1)$ TV distance) to $(T_k, \sigma_{T_k})$.
\end{lemma}
\begin{proof}
  This is a well-known result and has appeared in many places. For a proof, see e.g.,~\cite[Lemma 6.2]{mossel2022exact}.
\end{proof}

In the following, we use a.a.s.~(asymptotically almost surely) to denote that a event happens with probability $1-o(1)$.
\begin{lemma}[No long range correlations] \label{lem:sbm-approx-cond-indep}
  Let $(X,G=(V,E))\sim \SBM(n,q,\frac an,\frac bn)$.
  Let $A=A(G),B=B(G),C=C(G) \subseteq V$ be a (random) partition of $V$ such that $B$ separates $A$ and $C$ in $G$.
  If $|A\cup B| = o(\sqrt n)$ a.a.s., then
  \begin{align}
    \bP(X_A | X_{B\cup C}, G) = (1\pm o(1)) \bP(X_A | X_B, G)~\aas
  \end{align}
\end{lemma}
\begin{proof}
  This is a special case of \cite[Prop.~5.6]{gu2023channel}. For completeness we present a self-contained proof here. The proof is a generalization of \cite[Lemma 6]{mossel2015reconstruction}.

  For $e=\{u,v\}\in \binom V2$, define
  \begin{align}
    \psi_e(G, X) &:= \left\{
      \begin{array}{ll}
        \frac {w_{X_u,X_v}}n, & \text{if}~e\in E,\\
        1-\frac {w_{X_u,X_v}}n, & \text{if}~e\not\in E,\\
      \end{array}
    \right. \\
    \text{where}\quad w_{i,j} &:= \left\{
      \begin{array}{ll}
        a, & \text{if}~i=j,\\
        b, & \text{if}~i\ne j.\\
      \end{array}
    \right.
  \end{align}
  Then
  \begin{align}
    \bP(G,X) := \bP(X) \bP(G|X) = q^{-n} \prod_{e\in \binom V2} \psi_e(G,X).
  \end{align}
  We partition $\binom V2$ into four parts. Define
  \begin{align}
    E_1 &:= \left\{e\in \binom V2: |e\cap A|=|e\cap C|=1\right\},\\
    E_2 &:= \left\{e\in \binom V2: |e\cap C|=0\right\},\\
    E_3 &:= \left\{e\in \binom V2: |e\cap A|=0, |e\cap C|\ge 1\right\}.
  \end{align}
  Then $E_1 \cup E_2 \cup E_3 = \binom V2$ is a partition of $\binom V2$.
  Define
  \begin{align}
    Q_i := Q_i(G, X) := \prod_{e\in E_i} \psi_e(G,X)\quad \forall i\in [3].
  \end{align}
  Then
  \begin{align}
    \bP(G,X) = q^{-n} Q_1 Q_2 Q_3.
  \end{align}

  We prove that $Q_1$ is approximately independent of $X_{B\cup C}$ a.a.s.
  Let $(\alpha_n)_{n\ge 0}$ be a deterministic sequence with $\alpha_n=\omega(\sqrt n)$ and $\alpha_n|A| = o(n)$ a.a.s.
  Define
  \begin{align}
    &\Omega := \{Y\in [q]^V: \left|N_i(Y) - \frac nq\right| \le \alpha_n \forall i\in [q]\}, \\
    &\Omega_U := \{Y\in \Omega: Y_U = X_U\},\\
    \text{where}\quad&N_i(Y) := \#\{v\in V: Y_v=i\}.
  \end{align}
  By concentration of $N_i(X)$, we have $X\in \Omega$ a.a.s.

  Note that $E_1\cap E = \emptyset$.
  So
  \begin{align}
    \nonumber Q_1 &= \prod_{u\in A, v\in C} \psi_{uv}(G,X) \\
    \nonumber &=\prod_{u\in A, v\in C} \left(1-\frac{w_{X_u,X_v}}{n}\right) \\
    &= (1+o(1)) \prod_{u\in A, v\in C} \exp\left(-\frac{w_{X_u,X_v}}{n}\right)~\aas
    \label{eqn:prop-sbm-approx-cond-indep-q1-step-1}
  \end{align}
  where the third step is because $|A|=o(\sqrt n)$ a.a.s.~and
  \begin{align}
    \exp\left(-\frac{w_{X_u,X_v}}{n}\right) = \left(1 + O\left(n^{-2}\right)\right)\left(1-\frac{w_{X_u,X_v}}{n}\right).
  \end{align}
  For every $u\in A$ and $X\in \Omega$, we have
  \begin{align}
    \nonumber &~\prod_{v\in C} \exp\left(-\frac{w_{X_u,X_v}}{n}\right) \\
    \nonumber =&~ \exp\left(-\sum_{v\in C}\frac{w_{X_u,X_v}}{n}\right) \\
    % \nonumber =&~ \exp\left(-\frac 1{(r-1)!}\sum_{v_1,\ldots,v_{r-1}\in C}\frac{w_{X_u,X_v}}{n} \pm O(n^{-1})\right) \\
    \nonumber =&~ \exp\left(-\sum_{i\in [q]} \frac{w_{X_u,i}}{n} \cdot \left(\frac nq \pm O(\alpha_n)\right) \pm O(n^{-1})\right) \\
    \nonumber =&~ \exp\left(-\frac 1q \sum_{i\in [q]} w_{X_u,i} \pm O(\alpha_n n^{-1})\right) \\
    =&~ \exp\left(-d \pm O(\alpha_n n^{-1})\right) ~\aas
    \label{eqn:prop-sbm-approx-cond-indep-q1-step-2}
  \end{align}
  % where last step is by Condition~\ref{prop:sbm-approx-cond-indep}.

  Combining \eqref{eqn:prop-sbm-approx-cond-indep-q1-step-1} and \eqref{eqn:prop-sbm-approx-cond-indep-q1-step-2} we get
  \begin{align}
    \nonumber Q_1 &= (1+o(1)) \exp\left(- d |A| \pm O(\alpha_n |A| n^{-1})\right)\\
    &= (1\pm o(1)) \exp\left(- d |A| \right) =: (1\pm o(1)) K(G) ~\aas
    \label{eqn:prop-sbm-approx-cond-indep-q1}
  \end{align}

  By \eqref{eqn:prop-sbm-approx-cond-indep-q1}, we have
  \begin{align}
    \bP(G,X) = (1\pm o(1)) q^{-n} K(G) Q_2 Q_3 ~\aas
  \end{align}
  Furthermore, for any $U=U(G)\subseteq V$ we have
  \begin{align}
    \nonumber \bP(G,X_U) &= (1\pm o(1))\bP(G,X_U,X\in\Omega)\\
    &= (1\pm o(1)) \sum_{Y\in \Omega_U}q^{-n} K(G) Q_2(G,Y) Q_3(G,Y) ~\aas
  \end{align}
  % Slightly abusing notation, we write $Q_2(G,Y)=Q_2(G,Y_{A\cup B})$ and $Q_3(G,Y)=Q_3(G,Y_{B\cup C})$.

  Therefore
  \begin{align}
    \bP(X_A | X_B, G) = \frac{\bP(X_{A\cup B}, G)}{\bP(X_B, G)}
    = (1\pm o(1)) \frac{\sum_{Y\in \Omega_{A\cup B}} Q_2(G,Y) Q_3(G,Y)}{\sum_{Y\in \Omega_B} Q_2(G,Y) Q_3(G,Y)}~\aas
    \label{eqn:prop-sbm-approx-cond-indep-condB-step}
  \end{align}
  Note that $Q_2(G,Y)$ is a function of $(G,Y_{A\cup B})$ and $Q_3(G,Y)$ is a function of $(G,Y_{B\cup C})$.
  So the numerator of \eqref{eqn:prop-sbm-approx-cond-indep-condB-step} is a.a.s.~equal to
  \begin{align}
    Q_2(G,X)\sum_{Y\in \Omega_{A\cup B}} Q_3(G,Y)
    \label{eqn:prop-sbm-approx-cond-indep-condB-step-1}
  \end{align}
  and the denominator of \eqref{eqn:prop-sbm-approx-cond-indep-condB-step} is a.a.s.~equal to
  \begin{align}
    \left(\sum_{Y\in \Omega_{B\cup C}} Q_2(G,Y)\right)
    \left(\sum_{Y\in \Omega_{A\cup B}} Q_3(G,Y)\right).
    \label{eqn:prop-sbm-approx-cond-indep-condB-step-2}
  \end{align}
  Combining \eqref{eqn:prop-sbm-approx-cond-indep-condB-step}\eqref{eqn:prop-sbm-approx-cond-indep-condB-step-1}\eqref{eqn:prop-sbm-approx-cond-indep-condB-step-2}, we get
  \begin{align}
    \bP(X_A | X_B, G)= (1\pm o(1)) \frac{Q_2(G,X)}{\sum_{Y\in \Omega_{B\cup C}} Q_2(G,Y)}~\aas
    \label{eqn:prop-sbm-approx-cond-indep-condB}
  \end{align}
  Similarly, we have
  \begin{align}
    \nonumber &~\bP(X_A | X_{B\cup C}, G) = \frac{\bP(X, G)}{\bP(X_{B\cup C}, G)} \\
    \nonumber =&~ (1\pm o(1)) \frac{Q_2(G,X) Q_3(G,X)}{\sum_{Y\in \Omega_{B\cup C}} Q_2(G,Y) Q_3(G,Y)} \\
    =&~(1\pm o(1)) \frac{Q_2(G,X)}{\sum_{Y\in \Omega_{B\cup C}} Q_2(G,Y)}~\aas
    \label{eqn:prop-sbm-approx-cond-indep-condBC}
  \end{align}
  Comparing \eqref{eqn:prop-sbm-approx-cond-indep-condB} and \eqref{eqn:prop-sbm-approx-cond-indep-condBC} we finish the proof.
\end{proof}

\section{Optimal recovery algorithm} \label{sec:opt-rec}
In this section we prove Theorem~\ref{thm:opt-rec-survey} and~\ref{thm:opt-rec-vanilla}.
\begin{algorithm}[!ht]\caption{Belief propagation algorithm for SBM with side information}\label{alg:opt-rec-survey}
  \begin{algorithmic}[1]
    \State \textbf{Input: } SBM graph $G=(V,E)$, side information $Y\in \cY^V$
    \State \textbf{Output: } $\wh X\in [q]^V$
    \State $m_{u\to v}^{(0)} \gets \left(\frac{W(Y_u | i)}{\sum_{j\in [q]} W(Y_u | j)}\right)_{i\in [q]} \quad \forall (u,v)\in E$
    \State $r \gets \lfloor \log^{0.9} n\rfloor $
    \For{$t=0\to r-1$}
    \For{$(u,v)\in E$}
    \State $$Z_{u\to v}^{(t+1)} \gets \sum_{j\in [q]} W(Y_u | j) \prod_{(u,w)\in E, w\ne v} \sum_{k\in [q]} m_{w\to u}^{(t)}(k) P_\lambda(k | j)$$
    \State $$m_{u\to v}^{(t+1)} \gets \left( \left(Z_{u\to v}^{(t+1)}\right)^{-1} \cdot W(Y_u | i) \prod_{(u,w)\in E, w\ne v} \sum_{k\in [q]} m_{w\to u}^{(t)}(k) P_\lambda(k | i)\right)_{i\in [q]}$$
    \EndFor
    \EndFor
    \For{$u\in V}$
    \State $$Z_u \gets \sum_{j\in [q]} W(Y_u | j) \prod_{(u,w)\in E} \sum_{k\in [q]} m_{w\to u}^{(r)}(k) P_\lambda(k | j)$$
    \State $$m_u \gets \left(Z_u^{-1} \cdot W(Y_u | i) \prod_{(u,w)\in E} \sum_{k\in [q]} m_{w\to u}^{(r)}(k) P_\lambda(k | i)\right)_{i\in [q]}$$
    \State $\wh X_u \gets \argmax_{i\in [q]} m_u(i)$
    \EndFor
    \State \Return $\wh X$
  \end{algorithmic}
\end{algorithm}

\begin{proof}[Proof of Theorem~\ref{thm:opt-rec-survey}]
  We run Algorithm~\ref{alg:opt-rec-survey}.
  Let $\rho \in V$ be a fixed vertex.
  For $k\in \bZ_{\ge 1}$, define $B(\rho,k)$, $\partial B(\rho,k)$ as in Lemma~\ref{lem:sbm-bot-coupling}.
  By Lemma~\ref{lem:sbm-bot-coupling} and induction on $t$, we see that $m_{u\to v}^{(t)}$ has the same distribution (up to $o(1)$ TV distance) as the posterior distribution of $\sigma_\rho$ conditioned $\omega_{T_t}$.
  Therefore $m_{u\to v}^{(r)}$ has the same distribution (up to $o(1)$ TV distance) as the posterior distribution of $\sigma_\rho$ conditioned $\omega_{T_r}$.
  So as $n\to \infty$, Algorithm~\ref{alg:opt-rec-survey} achieves accuracy
  \begin{align}
    1-\lim_{k\to \infty} P_e(\sigma_\rho | T_k, \omega_{T_k}).
  \end{align}

  On the other hand, we have
  \begin{align}
    \nonumber P_e(X_\rho | G, Y) &\ge P_e(X_\rho | G, Y, X_{\partial B(\rho,k)}) \\
    \nonumber &= P_e(X_\rho | G, Y_{B(\rho,k)}, X_{\partial B(\rho,k)}) \pm o(1) \\
    \nonumber &\ge P_e(\sigma_\rho | T_k, \omega_{T_k}, \sigma_{L_k}) - o(1). \\
    &=P_e(\sigma_\rho | T_k, \omega_{T_k}) - o(1).
  \end{align}
  where the first step is by data processing inequality,
  the second step is by Lemma \ref{lem:sbm-approx-cond-indep},
  the third step is by Lemma~\ref{lem:sbm-bot-coupling},
  the fourth step is by boundary irrelevance with respect to $W$.
  Taking limit $n\to \infty$, then $k\to \infty$, we see that
  \begin{align}
    P_e(X_\rho | G, Y) \ge \lim_{k\to \infty} P_e(\sigma_\rho | T_k, \omega_{T_k}).
  \end{align}
  This shows that Algorithm~\ref{alg:opt-rec-survey} is optimal.
\end{proof}

\begin{algorithm}[!ht]\caption{Belief propagation algorithm for SBM}\label{alg:opt-rec-vanilla}
  \begin{algorithmic}[1]
    \State \textbf{Input: } SBM graph $G=(V,E)$, initial recovery algorithm $\cA$
    \State \textbf{Output: } $\wh X\in [q]^V$
    \State $s\gets 1$ if model is assortative; $s\gets -1$ if model is disassortative
    \State $r \gets \lfloor \log^{0.9} n\rfloor $
    \State $U \gets $ random subset of $V$ of size $\lfloor \sqrt n\rfloor$
    \State $Y \gets \cA(G\backslash U)$
    \State For $i\in [q]$, $u\in U$, compute $N_Y(u,i) \gets \#\{Y_v=i : v\in V, (u,v)\in E\}$
    \State For $i\in [q]$, choose $u_i\in U$ such that  \label{line:alg-opt-rec-choose-u}
    \begin{enumerate}[label=(\alph*)]
      \item $u_i$ has at least $\sqrt{\log n}$ neighbors in $V\backslash U$, and \label{item:alg-opt-rec-choose-u-req-a}
      \item $s N_Y(u_i,i) > s N_Y(u_i,j)$ for $j\in [q]\backslash i$. \label{item:alg-opt-rec-choose-u-req-b}
    \end{enumerate}
    \For{$v\in V\backslash U$}
      \State $Y^v \gets \cA(G\backslash B(v,r-1) \backslash U)$
      \State Relabel $Y^v$ by performing a permutation $\tau \in \Aut([q])$, so that $s N_{Y^v}(u_i,i) > s N_{Y^v}(u_i,j)$ for $i\in [q]$, $j\in [q]\backslash i$. Report failure if this cannot be achieved.
      \State $M^v_{i,j} \gets \frac{N_{Y^v}(u_i,j)}{\sum_{j\in [q]}N_{Y^v}(u_i,j)}$
      \State Run belief propagation on $B(v,r-1)$ with boundary condition $Y^v_{\partial B(v,r)}$, assuming the channel from $\partial B(v,r-1)$ to $\partial B(v,r)$ is $H^v$ \label{line:alg-opt-rec-bp}
      \State $\wh X_v \gets $ maximum likelihood label according to belief propagation \label{line:alg-opt-rec-bp-result}
    \EndFor
    \State $\wh X_v \gets 1$ for all $v\in U$
    \State \Return $\wh X$
  \end{algorithmic}
\end{algorithm}

\begin{proof}[Proof of Theorem~\ref{thm:opt-rec-vanilla}]
  We run Algorithm~\ref{alg:opt-rec-vanilla}.
  The proof is a variation of the proof in \cite{mossel2016belief}.

  \textbf{Choice of $u_i$.} For every $i\in [q]$, the set $\{u\in U: X_u=i\}$ has size $\frac nq \pm o(n)$.
  Therefore with high probability, there exists $u\in U$ with $X_u=i$ that satisfies~\ref{item:alg-opt-rec-choose-u-req-a}.
  Furthermore, because $Y$ is independent of $U$, we can equivalently first generate the graph $G\backslash U$, then compute $Y$, then generate the edges adjacent to $U$.
  In this way, we see that with high probability, for all $u\in U$ satisfying~\ref{item:alg-opt-rec-choose-u-req-a}, the empirical distribution of $\{Y_v: v\in V, (u,v)\in E\}$ has $o(1)$ total variation distance to $P_\lambda F$.
  By assumption~\ref{item:thm-opt-rec-vanilla-req-1}\ref{item:thm-opt-rec-vanilla-req-3} in Theorem~\ref{thm:opt-rec-vanilla}, we have $s (P_\lambda F)_{i,\tau(i)} > s (P_\lambda F)_{i,\tau(j)} + |\lambda| \epsilon$ for $i\in [q]$, $j\in [q]\backslash i$.
  Therefore with high probability, for all $u\in U$ satisfying~\ref{item:alg-opt-rec-choose-u-req-a}, we can identify $X_u$ up to a permutation $\tau\in \Aut([q])$ by computing $\argmax_{j\in [q]} s N_Y(u,j)$.
  Therefore with high probability we are able to choose the $u_i$s in Line~\ref{line:alg-opt-rec-choose-u}.

  \textbf{Alignment of $Y$ with $Y^v$.} The above discussion still holds with $Y$ replaced by $Y^v$. One thing to note is that by Lemma~\ref{lem:sbm-bot-coupling}, $|B(v,r-1)| = n^{o(1)}$ with high probability. So removing $B(v,r-1)$ from $G$ has negligible influence to the the empirical distribution of labels of neighbors of $u_i$s.
  Therefore, with high probability, we are able to permute the labels $Y^v$ so that the empirical distributions align with that of $Y$.
  (Note that we do not assume the empirical distributions for $Y$ and $Y^v$ are the same; we only use that they both satisfy condition~\ref{item:thm-opt-rec-vanilla-req-3}.)
  Furthermore, we can compute the transition matrix
  \begin{align}
    M^v = P_\lambda F^v \pm o(1).
  \end{align}

  \textbf{Boundary condition of BP.}
  Because $Y^v$ is independent of edges between $\partial B(v,r-1)$ and $\partial B(v,r)$, we can equivalently first generate the graph $G\backslash B(v,r-1)\backslash U$, then compute $Y^v$, then generate $E(\partial B(v,r-1), \partial B(v,r))$.
  In this way, it is clear that $Y^v_w$ for one $(u,w)\in E(\partial B(v,r-1), \partial B(v,r))$ is equivalent to one observation of $X_u$ through channel $M^v$.

  \textbf{Property of $M^v$.}
  Note that $M^v \ge_{\deg} P_{\lambda-o(1)} F^v$.
  % By condition~\ref{item:thm-opt-rec-req-2}, $\sigma_{\min}(M^v) \ge |\lambda| \sigma_{\min}(F^v) - o(1) = |\lambda| \epsilon - o(1)$.
  By Lemma~\ref{lem:sing-deg-potts} and condition~\ref{item:thm-opt-rec-vanilla-req-2}, we have $F^v\ge_{\deg} P_{\lambda'}$ for some constant $\lambda'>0$ not depending on $n$.
  Therefore $M^v \ge_{\deg} P_{\lambda''}$ for some $\lambda''>0$ not depending on $n$.

  \textbf{Convergence of BP recursion.}
  Because $\lambda''>0$ is a constant, by the stability of BP fixed point assumption, for any $\kappa>0$, there exists some integer $k_0$ not depending on $n$ such that
  \begin{align}
    P_e(\BP^{k_0}(M^v)) \ge P_e(\BP^{k_0}(P_{\lambda''})) > \lim_{k\to \infty} P_e(\BP^k(\Id))-\kappa.
  \end{align}
  Because $r = \omega(1)$, belief propagation in Line~\ref{line:alg-opt-rec-bp} converges to $o(1)$ in TV distance to the fixed point.
  Therefore we achieve desired accuracy in Line~\ref{line:alg-opt-rec-bp-result}.
\end{proof}

\section{Discussions} \label{sec:discuss}
\paragraph{Phase transition for BP uniqueness and boundary irrelevance.}
We summarize the phase transition for BP uniqueness and boundary irrelevance as follows.
The important thresholds are the reconstruction threshold and the Kesten-Stigum threshold ($d\lambda^2=1$).
\begin{itemize}[nolistsep,leftmargin=*]
  \item Below the reconstruction threshold: The $\BP$ operator (without survey) has no non-trivial fixed points.
  We conjecture that BI always holds below the reconstruction threshold.
  Low SNR part in Theorem~\ref{thm:bi-imprecise} shows that BI holds whenever $d\lambda^2< q^{-2}$. % This is tight within a factor of at most $q^2$.
  \item Between the reconstruction and the KS threshold:
  Theorem~\ref{thm:non-bi} shows that boundary irrelevance does not hold in this regime.
  The BP operator has a non-trivial fixed point. This fixed point is not globally stable, i.e., there exists non-trivial channel $U$ such that $\BP^\infty(U)$ is trivial \cite{janson2004robust}.
  We do not know whether the BP operator has a unique non-trivial fixed point.
  \item Above the KS threshold: We conjecture that BP uniqueness and BI always hold in this regime, and the unique non-trivial fixed point is globally stable. High SNR part in Theorem~\ref{thm:bi-imprecise} is tight within a factor of $1+\log q$ for general $(\lambda,d)$, and asymptotically tight within a factor of $1+o_q(1)$ for $q\to \infty$ and $\lambda = o(\log q)$.
\end{itemize}

\paragraph{Asymmetric models.}
The Potts model is a very symmetric model and can be studied using FMS channels. For more general SBM and BOT models, the class of FMS channels is no longer suitable. Nevertheless, it might be possible to extend our degradation method to the asymmetric case to achieve better results than \cite{chin2021optimal}, which uses a generalization of the method of \cite{mossel2016belief}.

Fix a finite alphabet $\cX$ and a distribution $\pi$ on $\cX$ with full support.
Then a channel $\cP:\cX\to \cY$ can be viewed as a distribution of posterior distributions under prior $\pi$, i.e., the distribution of $P_{X|Y}$ where $P_X = \pi$, $Y \sim P(\cdot | X)$. Note that this is a distribution on $\cP(\cX)$.
Similarly to Prop.~\ref{prop:fms-deg-couple}, degradation between channels with input alphabet $\cX$ can be equivalently characterized as a coupling between the two distributions of posterior distributions.
Let $\phi$ be a strongly convex function on $\cP(\cX)$. Using the above distributional characterization of channels, we can extend $\phi$ to a potential function $\Phi$ on the space of channels with input alphabet $\cX$.
For two sequences $(M_k)_{k\ge 0}$, $(\wt M_k)_{k\ge 0}$ satisfying the BP recursion and related by degradation $\wt M_k \le_{\deg} M_k$, if $\lim_{k\to \infty} (\Phi(M_k)-\Phi(\wt M_k))=0$, then we should expect $M_\infty = \wt M_\infty$.

The difficulty in carrying out the above plan is to analyze behavior of $\Phi$ under BP recursion.
One also needs to keep in mind that the Potts model admits more than one non-trivial fixed points in the space of all $q$-ary input channels (Section~\ref{sec:asymm}). So extra assumptions are needed when dealing with the general case.

\paragraph{\cite{yu2022ising}'s degradation method.}
\cite{yu2022ising} proved uniqueness of BP fixed point and boundary irrelevance for all symmetric Ising models. Therefore a natural idea is to generalize their proof to the Potts model. However, it seems difficult to make this work.

While \cite{yu2022ising}'s proof is also based on degradation of BMS channels, their method is quite different from \cite{abbe2021stochastic}.
Specifically, \cite{yu2022ising} uses a stronger version of the stringy tree lemma of \cite{evans2000broadcasting}, which says that $(P^{\star d})\circ \BSC_\delta \le_{\deg} (P\circ \BSC_\delta)^{\star d}$ for all BMS channels $P$, and under additional assumptions one even has $(P^{\star d})\circ \BSC_\delta \le_{\deg} (P\circ \BSC_\delta)^{\star d} \circ \BSC_\epsilon$ for some $\epsilon>0$ depending on $P,\delta,d$.
The natural generalization of the original stringy tree lemma to the Potts model (with $\BSC_\delta$ replaced by $P_\lambda$ and BMS channels replaced by FMS channels) is not true, even for $d\lambda^2>1$. Therefore it is unclear how to generalize \cite{yu2022ising}'s proof to the Potts model.

\paragraph{Robust reconstruction.}
The boundary irrelevance problem is related to the robust reconstruction problem, which asks whether the trivial fixed point of the BP operator is locally stable, i.e., whether for weak enough initial channel $U$, we have $\BP^\infty(U)=0$.
In fact, our proof of Theorem~\ref{thm:non-bi} can be slightly modified to show that the Potts model does not admit robust reconstruction below the Kesten-Stigum threshold.

The robust reconstruction problem for BOT models has been extensively studied in \cite{janson2004robust}, which showed that the robust reconstruction threshold is at the KS threshold, i.e., the model admits robust reconstruction when $d\lambda^2>1$ and does not when $d \lambda^2 < 1$.
Their proof uses the contraction of a $\chi^2$-like information measure, which can take value $\infty$ in some corner cases, e.g., for the coloring model with erasure leaf observations, or for the boundary irrelevance problem. Furthermore, they only considered trees with bounded maximum degree, which does not include the Poisson tree.
Therefore to prove Theorem~\ref{thm:non-bi} we cannot directly use \cite{janson2004robust}.

\end{document}